\documentclass[10pt]{article}

\usepackage[margin=1.9cm]{geometry}
\usepackage[english]{babel}
\usepackage{amssymb,amsthm,amsmath}
\usepackage{graphicx}
\usepackage{multirow}
\usepackage{subfigure}
\def \b {\beta}
\def \d {\delta}

\def \e {\varepsilon}

\def \o {\omega}
\def \OO {\Omega}
\def \OS {\Omega^*}
\def \t {\tau}

\def \g {\gamma}
\def \CE {\mathcal{E}}
\def \E {\mathbb E}
\def \N {\mathbb N}
\def \R {\mathbb R}

\def \B {\mathcal{B}}
\def \LL {\mathcal{L}}
\def \ninf {\nu \to \infty}
\def \limnu {\lim_{\ninf}}
\def \liminfnu {\liminf_{\ninf}}
\def \limsupnu {\limsup_{\ninf}}
\def \ed {\,{\buildrel d \over =}\,}
\def \cd {\xrightarrow{d}}

\def \NA {\mathcal{N}}
\def \A {\mathcal{A}}
\def \SA {\mathcal{S}}

\def \tm {t_{\mathrm{mix}}}
\def \tr {\tau^{\mathrm{res}}}
\def \xt {X(t)}
\def \xtb {\{\xt \}_{t \geq 0}}
\def \xst {X^*(t)}
\def \xstb {\{\xst \}_{t \geq 0}}
\def \rmexp {\mathrm{Exp}}
\def \LT {\mathcal{L}}
\def \TT {T_{(k_1,l_1),(k_2,l_2)}}
\def \TTo {T_{(k_1,l_1),(k_2,1)}}

\newcommand{\EE}[1]{\mathbb{E}\left\{ #1 \right\}}
\newcommand{\bin}[2]{\begin{pmatrix} #1 \\ #2 \end{pmatrix}}
\newcommand{\pr}[1]{\mathbb P \left ( #1\right  )}
\newcommand{\eqn}[1]{\begin{equation} #1 \end{equation}}
\newcommand{\eqan}[1]{\begin{align} #1 \end{align}}

\newtheorem{thm}{Theorem}[section]
\newtheorem*{thma}{Theorem}

\newtheorem{lem}[thm]{Lemma}
\newtheorem{prop}[thm]{Proposition}
\newtheorem*{nnlem}{Lemma}

\theoremstyle{remark}
\newtheorem{rem}{Remark}

\usepackage{microtype}
\DisableLigatures{encoding = *, family = * }

\begin{document}
\title{Slow transitions, slow mixing, and starvation in dense random-access networks}
\author{A.~Zocca\thanks{Department of Mathematics and Computer Science, Eindhoven University of Technology} \and S.C.~Borst\footnotemark[1] \thanks{Alcatel-Lucent Bell Labs, Murray Hill, NJ, USA} \and J.S.H.~van Leeuwaarden\footnotemark[1]}
\date{}

\maketitle

\begin{abstract}
We consider dense wireless random-access networks, modeled as systems of particles with hard-core interaction. The particles represent the network users that try to become active after an exponential back-off time, and stay active for an exponential transmission time. Due to wireless interference, active users prevent other nearby users from simultaneous activity, which we describe as hard-core interaction on a conflict graph. We show that dense networks with aggressive back-off schemes lead to extremely slow transitions between dominant states, and inevitably cause long mixing times and starvation effects.
\end{abstract}

\noindent \textbf{Keywords:} wireless random-access networks; hitting times; throughput analysis; starvation phenomena; mixing times.

\section{Introduction}
\label{sec1}
We consider a stylized model for a network of~$N$ users sharing a wireless medium according to a random-access scheme. The network is represented by an undirected graph~$G = (V, E)$, called \textit{conflict graph}. The set of vertices~$V = \{1, \dots, N\}$ describes the network users and the set of edges~$E \subseteq V \times V$ indicates which pairs of users interfere and are thus prevented from simultaneous activity. The independent sets of~$G$ (sets of vertices not sharing any edge) then correspond to the feasible joint activity states of the network.

A user is said to be \textit{blocked} whenever the user itself or any of its neighbors in~$G$ is active, and \textit{unblocked} otherwise. User~$i$ activates (starts a transmission) at an exponential rate~$\nu_i$ whenever it is unblocked, and then remains active for an exponentially distributed time period with unit mean, before turning inactive again. The durations of the various activity periods are assumed independent across time and among users. We will refer to the parameters $\nu_i$ as \textit{activation rates}.

Let~$\OS \subseteq \{0, 1\}^V$ be the collection of incidence vectors of all independent sets of~$G$, and let~$\xst \in \OS$ be the joint activity state at time~$t$, with element $i$ of $\xst$ indicating whether user~$i$ is active ($X_i^*(t)=1$) or not ($X_i^*(t)=0$) at time~$t$. Then $\xstb$ is a reversible Markov process with stationary distribution~\cite{BKMS87,Kelly79,Kelly85,KBC87,WK05}
\eqn{\label{eq:statt1}
\pi_x(\nu_1, \dots, \nu_N) = \lim_{t \to \infty} \pr{\xst = x} =
\frac{\prod_{i \in V} \nu_i^{x_i}}{\sum_{y \in \OS} \prod_{i \in V} \nu_i^{y_i}},
\quad x \in \OS.
}
We also mention that the model amounts to a special instance of a loss network~\cite{SR04,ZZ99}, and that the product-form distribution~\eqref{eq:statt1} corresponds to the Gibbs measure of the hard-core model in statistical physics~\cite{GS08,H97}.

For the case $\nu_i=\nu$ it follows from \eqref{eq:statt1} that only the activity states corresponding to \textit{maximum} independent sets retain probability mass as the activation rate~$\nu$ grows large. This indicates that users that do not belong to a maximum independent set have far fewer opportunities to be active. This disadvantage is commonly referred as \textit{spatial unfairness}, and the associated starvation effects have major performance repercussions in wireless networks.

It has been shown that spatial unfairness can be avoided by selecting suitable user-specific activation rates~$\nu_i$ which provide all users with an equal opportunity to be active in the long run~\cite{JW10,VJLB11}. Even in those cases, however, or in symmetric scenarios where spatial fairness is automatically ensured, transient but yet significant starvation effects can arise due to extremely slow transitions between high-likelihood or \textit{dominant} states. Intuitively speaking, the activity process will typically need to pass through a low-likelihood or bottleneck state in order for the process to transit between dominant states. Visiting such a bottleneck state basically involves the occurrence of a rare event, or even several rare events in different limiting regimes, and causes the transition to take a correspondingly long amount of time. Consequently, users may experience extended stretches of forced inactivity (possibly interspersed by long intervals with a rapid succession of activity periods), resulting in serious performance degradation. 

Motivated by these fairness issues, we investigate in the present paper the time for the Markov process to reach, starting from a given dominant state, one of the other dominant states. We study these hitting times as well as mixing properties in the asymptotic regime where the activation rates~$\nu_i$ grow large. This asymptotic regime, in which users activate aggressively, is relevant in highly loaded networks and gives rise to the above-described starvation effects. As shown numerically in~\cite{KKA12}, these starvation effects are particularly pronounced in dense topologies. As a prototypical worst-case scenario, we focus on a specific class of dense conflict graphs, namely complete partite graphs. In such networks the users can be partitioned into $K$~disjoint sets called \textit{components}, such that each user interferes with all users in all other components. This implies that a transition from an activity state in one of the components to another component entails passing through a bottleneck state where all users are inactive at some point. Based on this observation and a regenerative argument, we establish a geometric-sum representation for the hitting time, which we then use to obtain the asymptotic order-of-magnitude and scaled distribution. For convenience we assume all users within a given component to have the same activation rate, but we do allow for users in different components to have different activation rates. Section~\ref{sec2} presents a detailed model description and Section~\ref{sec3} gives an overview of the main results. Preliminary results for the case $\nu_i = \nu$ for all $i \in V$ appeared in~\cite{ZBvL12}, but did not exhibit the full qualitative range of asymptotic behaviors that will be revealed in the present paper.

\section{Model description}
\label{sec2}
Consider a network represented by a \textit{complete partite graph}, where two users interfere if and only if they belong to different components. Thus, in particular, users within the same component do not interfere. Denote by $C_1,\dots,C_K$ the $K$ components of $G$ and define $L_k:=|C_k|$ as the size of component $C_k$. Note that the components $C_1, \dots, C_K$ are the $K$ \textit{maximal} independent sets of the graph $G$. Moreover, component $C_k$ corresponds to a \textit{maximum} independent set if and only if $L_k \geq L_j$ for all $j=1,\dots,K$. Figure~\ref{fig:conflictgraph} shows an example of such a dense conflict graph, where $K=5$ and the components have sizes $\{ L_1, L_2, L_3, L_4, L_5\}=\{3,4,6,2,5\}$. The corresponding state space $\Omega^*$ for this graph is shown is Figure~\ref{fig:statespace}.

\begin{figure}[!hb]
\centering
\includegraphics[scale=0.3]{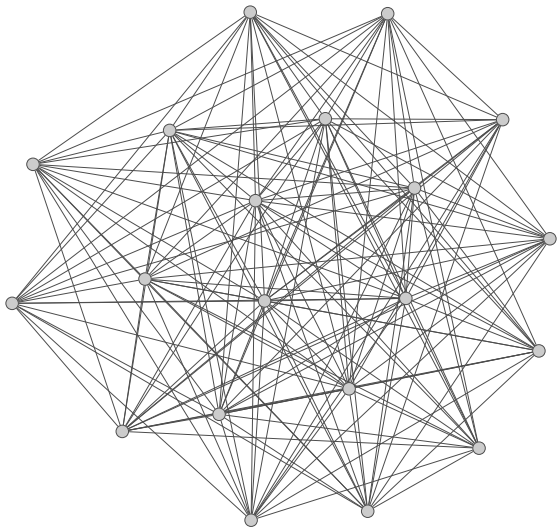}
\caption{Example of complete $K$-partite conflict graph with $K=5$.}
\label{fig:conflictgraph}
\end{figure}

We assume that the exponential rate at which a user activates depends only on a global aggressiveness parameter $\nu$ and on the component it belongs to, namely
\[
\nu_i = f_k(\nu) \quad \text{ if } i \in C_k,
\]
for some monotone function $f_k: \R_+ \to \R_+$ with $\limnu f_k(\nu) = \infty$. We will refer to the function $f_k(\cdot)$ as the \textit{activation rate} of component $C_k$, for $k=1,\dots,K$.

In view of symmetry, all states with the same number of active users in a given component can be aggregated, and we only need to keep track of the number $l$ of active users, if any, and the index $k$ of the component $C_k$ they belong to. This state aggregation yields a new Markov process $\xtb$ on a star-shaped state space~$\OO$ with $K$~branches, where each branch emanates from a common root node and describes one of the components of the conflict graph. Figure~\ref{fig:star} shows the aggregated state space corresponding to the previous example.

\begin{figure}[!hb]
\centering
\subfigure[State space $\Omega^*$]{\includegraphics[scale=0.36]{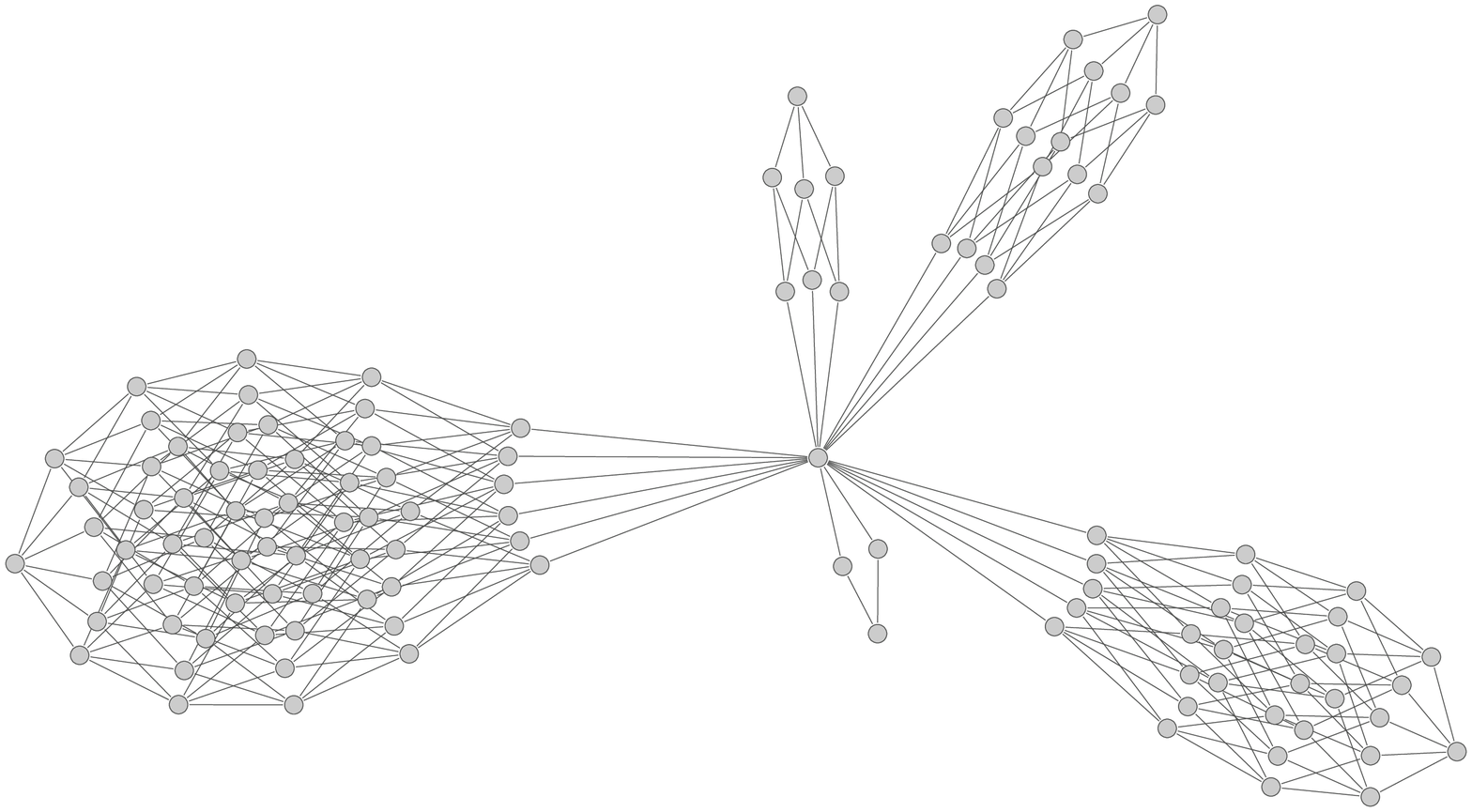}\label{fig:statespace}}
\subfigure[Aggregated state space $\Omega$]{\includegraphics[angle=181,origin=c,scale=0.32]{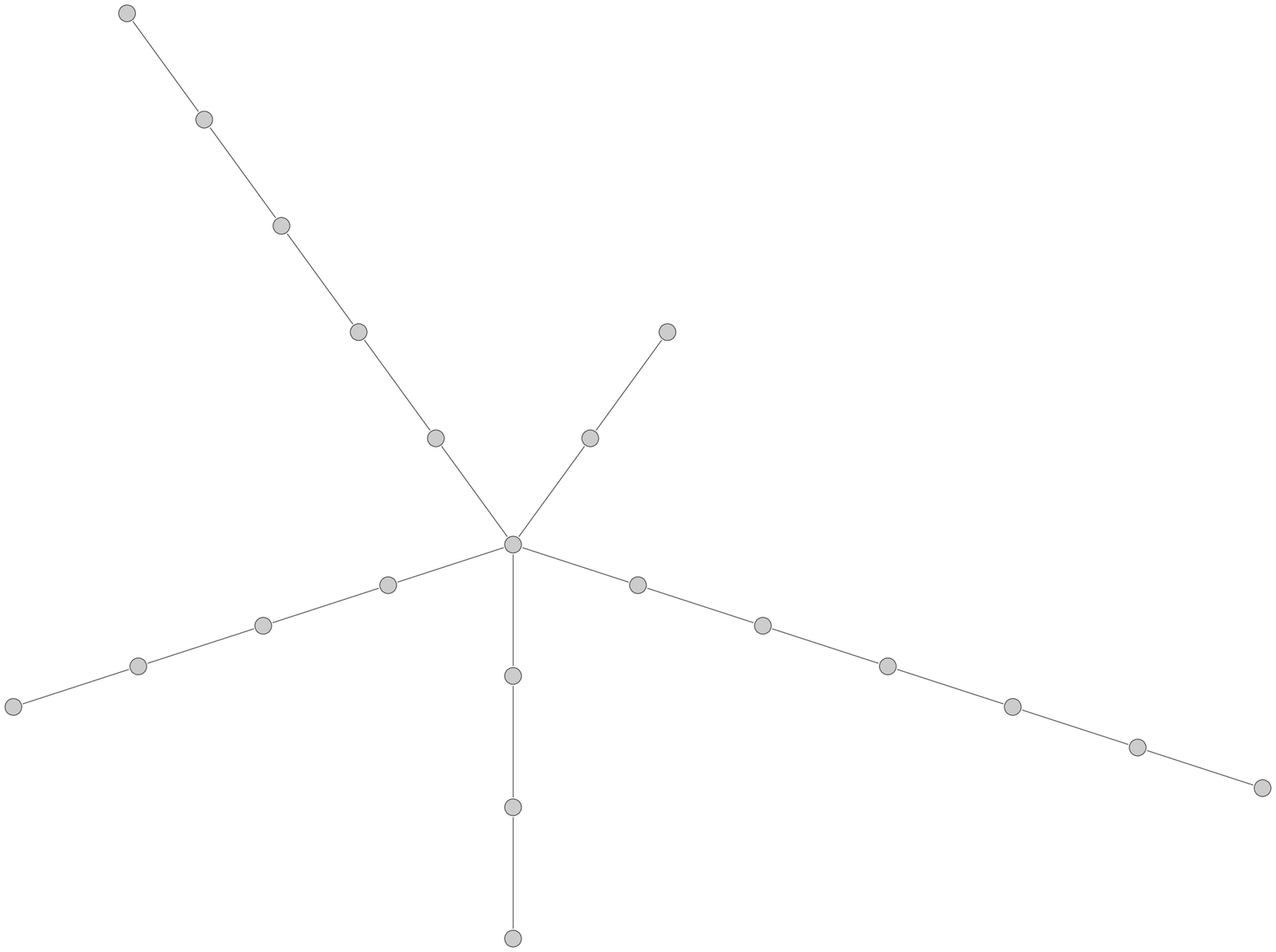}\label{fig:star}}
\caption{State space $\Omega^*$ and aggregated state space $\OO$, for the conflict graph in Figure~\ref{fig:conflictgraph}.}
\end{figure}

For $k=1,\dots, K$, let
\[
\B_k:=\{(k,l) : 1 \leq l \leq L_k\}
\]
denote  the branch of the state space $\OO$ that corresponds to activity inside component $C_k$, where state $(k,l)$ indicates $l$~users active in component $C_k$. Then $\OO = \{0\} \cup \bigcup_{k=1}^K \B_k$, where~$0$ is the bottleneck state in which all users are inactive.

The transition rates of the process $\xtb$ then read
\eqan{
q(0, (k, 1)) &= L_k f_k(\nu),\nonumber \\
q((k,1), 0) &= 1, \nonumber \\
q((k,l), (k,l+1)) &= (L_k - l) f_k(\nu), \quad l = 1, \dots, L_k - 1, \, k = 1, \dots, K,\nonumber \\
q((k,l), (k,l-1)) &= l, \quad l = 2, \dots, L_k, \, k = 1, \dots, K. \nonumber
}
The stationary distribution of the process $\xtb$ can be easily seen to be
\eqan{
\label{eq:sd}
\pi_0(\nu) &= \Big(1 + \sum_{k = 1}^{K} \sum_{l = 1}^{L_k}  \bin{L_k}{l} f_k(\nu)^l\Big)^{-1}, \nonumber \\
\pi_{(k,l)}(\nu) &= \pi_0(\nu)  \bin{L_k}{l} f_k(\nu)^l, \quad l = 1, \dots, L_k, \, k = 1, \dots, K.
}
The state $(k, L_k)$ corresponds to the maximum activity state inside component $C_k$, which becomes the most likely state within the branch $\B_k$ as $\ninf$.
Define the \textit{transition time} from state $(k_1,l_1)$ to state $(k_2,l_2)$ as
\[
\TT(\nu):=\inf\{t > 0: \xt = (k_2,l_2) ~|~ X(0) = (k_1,l_1)\}.
\]
We now introduce few parameters that will turn out to play a key role in the asymptotic distribution of the transition time. Define for $k\neq k_2$,
\eqn{\label{eq:defgk}
\g_k := \limnu \frac{f_k(\nu)^{L_k}}{\sum_{j\neq k_2} f_j(\nu)^{L_j}}.
}
To avoid technicalities, we assume throughout that all parameters $\g_k$ are well defined. In view of~\eqref{eq:sd}, $\g_k$ may be interpreted as the stationary fraction of time that the activity process spends in branch $k$ as $\ninf$, excluding the target branch $k_2$. As it turns out, $\g_k$ also equals the fraction of time that the activity process spends in branch $\B_k$ during the transition time as $\ninf$. 

Branch $\B_k$ is called \textit{dominant} if $\g_k >0$ and let $K_* :=\{k \neq k_2 : \g_k>0\}$ be the index set of all \textit{dominant branches}. Note that, by construction, the set $K_*$ is never empty and thus there is always at least one dominant branch.

\section{Main results}
\label{sec3}
In this section we present our main results, which are all related to the asymptotic behavior of the transition time $\TT(\nu)$ in the asymptotic regime of a large activation rate $\nu$.

Our first result characterizes the asymptotic order-of-magnitude of the mean transition time in terms of the activation rates and the network structure. For any two real-valued functions $f(\cdot)$ and $g(\cdot)$, let $f(\nu) \sim g(\nu)$ indicate that $\limnu f(\nu) / g(\nu) = 1$ as $\ninf$.

\begin{thm}\label{thm:thm1}
If $k_1\neq k_2$, then
\eqn{ \label{eq:agr}
\E T_{(k_1, l_1), (k_2, l_2)}(\nu) \sim \frac{1}{L_{k_1}} f_{k_1}(\nu)^{L_{k_1}-1} + \frac{1}{L_{k_2} f_{k_2}(\nu)} \sum_{k \in K_*} f_k(\nu)^{L_k}, \quad \text{ as } \ninf.
}
\end{thm}
The first term on the right-hand side of~\eqref{eq:agr} corresponds to the asymptotic mean \textit{escape time} $\E T_{(k_1,l_1),0} (\nu)$ from the initial branch $\B_{k_1}$, while the second term describes the contribution of the mean time spent visiting dominant branches, possibly including branch $\B_{k_1}$ as well. Let
\eqn{\label{eq:alpha}
\alpha:=\limnu \frac{\E T_{(k_1,l_1), 0} (\nu)}{\E \TT(\nu)} \in [0,1]
}
denote the relative weight of $\B_{k_1}$.

Our second result gives the asymptotic distribution of the transition time $\TT(\nu)$ scaled by its mean as $\ninf$.
\begin{thm}\label{thm:thm2}
If $k_1\neq k_2$, then
\[
\frac{\TT(\nu)}{ \E \TT (\nu)} \cd Z, \quad  \text{ as } \ninf.
\]
\end{thm}
The random variable $Z$ can be expressed as
\[Z \ed \alpha Y + (1-\alpha) W, \]
where the random variable $Y$ is exponentially distributed with unit mean and the random variable $W$ is independent of $Y$ and has a more complicated distribution, see~\eqref{eq:LTW}, which depends on the sizes and activation rates of the dominant branches only. The possible distributions of $Z$ are summarized in Table~\ref{tab:overview} in Section~\ref{sec5}. In several cases the distribution of $Z$ is exponential, which may be expected in view of the connection with many exponentiality results for the occurrence of rare events \cite{A82,AB92,AB93,GR05,K66,K79}. In addition, we identify various cases that lead to \textit{non}-exponentiality, typically due to the fact that the activity process spends a substantial period in branches other than $k_1$ and $k_2$.

Our third result concerns the starvation phenomenon. For $k=1, \dots, K$, define the random variable
\[
\tau_k(t):=\int_0^t  I_{\{X(s) \in \B_k \}} ds,
\]
which measures how much time the activity process $\xtb$ spends in branch $\B_k$ during the interval $[0,t]$. We can think of $\tau_k(t)$ as a measure of the \textit{throughput} of component $C_k$ over the time interval $[0,t]$. We speak of \textit{complete starvation} or \textit{zero throughput} of component $C_k$ in $[0,t]$ when $\tau_k(t)=0$. The next theorem provides insight into the time scales at which throughput starvation occur for a component of the network.

\begin{thm}\label{thm:thm3}
Assume $X(0)=(k_1,l_1)$ and $k_2 \neq k_1$. If $t(\nu) \sim \o \E \TTo(\nu)$, with $\o \in \R \cup \{0\}$, then
\eqn{\label{eq:starvation}
\limnu \pr{\tau_{k_2}(t(\nu)) = 0} \geq \pr{Z \geq \o}.
}
In particular, if $ t(\nu)=o(\E \TTo(\nu))$, then
\[
\limnu \pr{\tau_{k_2}(t(\nu)) = 0} =1,
\]
i.e. all users in $C_{k_2}$ have zero throughput for a period of length $t(\nu)$ with probability one as $\ninf$.
\end{thm}

The limit in~\eqref{eq:starvation} says that, even if there is long-term fairness among the components, for large values of $\nu$ all users in $C_{k_2}$ will face starvation on all time scales smaller than the mean transition time from the initial component to $C_{k_2}$.

For the Markov process at hand, slow transitions and starvation effects are intimately related with the mixing time. The mixing time of a process is a characterization of the time required for the process to reach equilibrium. Indeed, due to the complete partite structure of the conflict graph, the process is bound to be stuck in one of the dominant branches, leading to slow convergence to equilibrium. In Section~\ref{sec8} we define the mixing time in terms of the total variation distance from stationarity, and prove a lower bound for a large enough activation parameter $\nu$. This lower bound (see Proposition \ref{prop:mix}) indicates that the mixing time of the process is at least as large as the mean escape time from the dominant branch, which establishes a direct connection between transition times and mixing times.

\subsection*{Structure of the paper}

The remainder of the paper is organized as follows. In Section~\ref{sec4} we study the activity process within a single component, where it behaves as a birth-and-death process, bringing the asymptotic behavior within the realm of classical results. In Section~\ref{sec5} we then leverage these results in conjunction with a geometric-sum representation to prove both Theorems~\ref{thm:thm1} and~\ref{thm:thm2}. In Section~\ref{sec6} we sketch how the approach extends to scenarios where some of the users within the same component may interfere as well, relying on the same geometric-sum representation, but using more general asymptotic exponentiality results in~\cite{GR05} for the single-component behavior. We prove Theorem~\ref{thm:thm3} and a complementary result for throughput ``near-saturation'' in Section~\ref{sec7}. In Section~\ref{sec8} we derive a lower bound on the mixing time. Lastly, in Section~\ref{sec9} we make some concluding remarks and sketch some directions for further research.

\section{Hitting times within a single branch}
\label{sec4}

We first present a few results for the case where the two states $(k_1,l_1)$ and $(k_2,l_2)$ belong to the same branch, i.e.~$k_1 = k_2$, and $l_1 > l_2$. In this case, the presence of the other components does not affect the transition time, and hence we focus on a single branch, dropping the component index until further notice. Within a single component of size $L$, the process $\xtb$ evolves as an elementary birth-and-death process on the state space $\{L, L-1, \dots, 1,0\}$, so we can exploit several classical results for such processes. If we denote by $f(\nu)$ the activation rate for this component as a function of $\nu$, then the transition rates read
\eqan{
q(l,l+1)&=(L-l) f(\nu), \quad l=0,\dots , L-1, \nonumber \\
q(l,l-1)&=l, \quad l=1,\dots , L. \nonumber
}

\subsection{Asymptotic growth rate}
We first show how the mean transition time scales with the aggressiveness parameter $\nu$. 
\begin{prop}\label{prop:meansc}
For $L \geq l_1 > l_2 \geq 0$,
\[
\E T_{l_1,l_2}(\nu) \sim \frac{l_2! (L-l_2-1)! }{L!} f(\nu)^{L-l_2-1}, \quad  \text{ as } \ninf.
\]
\end{prop}
\begin{proof}
First observe that $\E T_{l_1,l_2}(\nu) = \sum_{l=l_1}^{l_2+1} \E T_{l,l-1}(\nu)$, so we can exploit a general result for birth-and-death processes~\cite{K65}, which in the present case says that, for $0 < l \leq L$,
\[
\E T_{l,l-1}(\nu) =
\frac{1}{l} \sum_{n=l}^{L} \frac{\pi_n(\nu)}{\pi_l(\nu)}.
\]
Now~\eqref{eq:sd} implies that $\pi_n(\nu) = o(\pi_L(\nu))$ as $\ninf$ for all $n = l, \dots, L - 1$, so that
\[
\E T_{l,l-1}(\nu) \sim \frac{1}{l} \frac{\pi_{L}(\nu)}{\pi_{l}(\nu)} =\frac{(l-1)! (L-l)!}{L!} f(\nu)^{L-l}, \quad  \text{ as } \ninf.
\]
Thus $\E T_{l,l-1}(\nu) = o(\E T_{l_2+1,l_2}(\nu))$ as $\ninf$ for all $l = l_1,\dots, $ $l_2$, and hence $\E T_{l_1,l_2}(\nu) \sim \E T_{l_2+1,l_2}(\nu)$ as $\ninf$ and the result follows.
\end{proof}

In order to gain insight in starvation effects, we are particularly interested in the time for the activity process to reach the center state~0, referred to as \textit{escape time}, because at such points in time users in other components have an opportunity to activate. Proposition~\ref{prop:meansc} shows that
\eqan{
\label{eq:sdp}
\E T_{l_1,0}(\nu) \sim \frac{1}{L} f(\nu)^{L-1}, \quad  \text{ as } \ninf.
}
Hence, the mean escape time grows asymptotically as a power of~$f(\nu)$, with the exponent equal to the component size minus one, and independent of the starting state~$l_1$.

\subsection{Asymptotic exponentiality}
We now turn to the scaled escape time, and show that it has an asymptotically exponential distribution. We will leverage the following well-known result for birth-and-death processes, which is commonly attributed to Keilson~\cite{K71} or Karlin and McGregor~\cite{KMcG59}.

\begin{thm}
\label{thm:km}
Consider a birth-and-death process with generator matrix~$Q$ on the state space $\{0,\dots,L\}$ started at state~$L$. Assume that $0$ is an absorbing state, and that the other birth rates $\{\lambda_i\}_{i=1}^{L-1}$ and death rates $\{\mu_i\}_{i=1}^L$ are positive. Then the absorption time in state~$0$ is distributed as the sum of $L$~independent exponential random variables whose rate parameters are the $L$~nonzero eigenvalues of $-Q$.
\end{thm}

Let $Q(\nu)$ be the generator matrix of the birth-and-death process $\xtb$ on the state space $\{L,L-1,\dots,1,0\}$, with $0$ an absorbing state.
Let $\{\theta_i(\nu)\}_{i=1}^L$ denote the non-zero eigenvalues of $-Q(\nu)$. It is known~\cite{LR54} that these eigenvalues are distinct, real and strictly positive, so we denote $0 < \theta_1(\nu) <  \theta_2(\nu) < \dots < \theta_L(\nu)$.

Theorem~\ref{thm:km} gives
\eqn{\label{eq:sumexp}
T_{L,0}(\nu) \ed \sum_{i=1}^{L} Y_i(\nu),
}
with $Y_1(\nu),\dots,Y_L(\nu)$ independent and exponentially distributed random variables with $\E Y_i(\nu)=1/\theta_i(\nu)$.

The following lemma relates the growth rates of the eigenvalues as $\ninf$ to the mean escape time $\E T_{L,0}(\nu)$.

\begin{lem}
\label{lem:asymptotic1}
$\limnu \theta_i(\nu) \cdot \E T_{L,0}(\nu) = 1$ if $i=1$ and $\infty$ if $i=2,\dots, L$.
\end{lem}
The proof of Lemma~\ref{lem:asymptotic1} is presented in~\ref{ap1}, and exploits detailed information about the growth rates of the eigenvalues obtained via symmetrization and the Gershgorin circle theorem. Lemma~\ref{lem:asymptotic1} shows that the smallest eigenvalue $\theta_1(\nu)$ becomes dominant as $\ninf$, but also proves the asymptotic exponentiality of the escape time. Indeed, denoting by $\LT_{X}(s)=\E (e^{-s X})$, with $\mathrm{Re}(s)>0 $, the Laplace transform of a random variable $X$,~\eqref{eq:sumexp} gives
\[
\LT_{T_{L,0}(\nu)/ \E T_{L,0}(\nu)}(s)=\prod_{i=1}^L \Big(1+\frac{s}{\theta_i(\nu) \cdot \E T_{L,0}(\nu)} \Big)^{-1}.
\]
Lemma~\ref{lem:asymptotic1} implies that
\[
\limnu \LT_{T_{L,0}(\nu) / \E T_{L,0}(\nu)}(s) = \frac{1}{1+s}.
\]
The continuity theorem for Laplace transforms then yields that the scaled escape time has an asymptotically exponential distribution as stated in the next theorem, where $\rmexp(\lambda)$ denotes an exponentially distributed random variable with mean $1/\lambda$.
\begin{thm}
\label{thm:expo}
\[
\frac{T_{L,0}(\nu)}{\E T_{L,0}(\nu)} \cd \rmexp(1), \quad  \text{ as } \ninf.
\]
\end{thm}

This result can be understood as follows. For large~$\nu$, the probability of hitting state~0 before the first return to state~$L$ becomes small. So the time $T_{L,0}(\nu)$ consists of a geometrically distributed number of excursions from~$L$ which return to~$L$ without hitting~0, followed by the remaining part of the excursion that hits~0. Hence, apart from this final part, $T_{L,0}(\nu)$ is the sum of a large geometrically distributed number of i.i.d.~random variables, which indeed is expected to be exponential.

The fact that the time until the first occurrence of a rare event is asymptotically exponential, is a widely observed phenomenon~\cite{K79}. Exponentiality of the hitting time of some subset~$B$ of the state space typically arises when the probability of hitting~$B$ in a single regenerative cycle is `small', and the cycle lengths are `not too heavy tailed'~\cite{GR05,K79}. This is also true for our situation, and hence an alternative proof of Theorem~\ref{thm:expo} can be obtained using~\cite[Thm.~1]{GR05} (which is a generalized version of~\cite{K66}). We do not use the probabilistic approach in~\cite{GR05} here, because the special case of a birth-and-death process allows for explicit analysis. However, in Section~\ref{sec6} we will discuss how this probabilistic approach can be exploited when the individual components have a more general structure and cannot be described by birth-and-death processes.

Let us finally remark that for reversible Markov processes similar exponentiality results were established in \cite{A82}-\cite{AB93}. Aldous~\cite{A82} showed that a result like Theorem~\ref{thm:expo} can be expected when the underlying Markov process converges rapidly to stationarity. This is indeed the case for the Markov process $\xtb$ restricted to a single branch.

To extend Theorem~\ref{thm:expo} to the case of a general starting state $0 < l \leq L$, we need the following technical lemma, whose proof is given in~\ref{ap2}.

\begin{lem}
\label{lem:asymbounds}
Let $T(\nu), U(\nu),$ $V(\nu),W(\nu)$ be non-negative random variables. Consider the properties
\begin{itemize}
\item[{\rm (i)}] $\limnu\E V(\nu)/ \E U(\nu) =\limnu \E W(\nu)/ \E U(\nu)= 0$.
\item[{\rm (ii)}] For every $\nu > 0$, $U - V \le_{st} T \le_{st} U + W$, i.e.~for every $ t > 0$,
\[\pr{U - V > t} \leq \pr{T > t} \leq \pr{U + W > t}.\]
\item[{\rm (iii)}] $U(\nu) /\E U(\nu) \cd Z$ as $\ninf$, where $Z$ is a continuous random variable independent of~$\nu$.
\end{itemize}
Then,
\begin{itemize}
\item[{\rm (a)}] If {\rm (i)} and {\rm (ii)} hold, then $\limnu \E T(\nu)/\E U (\nu) =1.$
\item[{\rm (b)}] If {\rm (i)}, {\rm (ii)} and {\rm (iii)} hold, then $T(\nu) / \E T(\nu) \cd Z$, as $\ninf.$
\end{itemize}
\end{lem}

\begin{prop}
\label{prop:ael0}
For any $ 0 < l \leq L$,
\[
\frac{T_{l,0}(\nu)}{\E T_{l,0}(\nu)} \cd \rmexp(1), \quad  \text{ as } \ninf.
\]
\end{prop}
\begin{proof}
The birth-and-death structure of the process and the strong Markov property yield the stochastic identity $T_{L,0}(\nu) \ed T_{L,l}(\nu)+T_{l,0}(\nu)$, which gives the stochastic bounds
$
T_{L,0}(\nu) - T_{L,l}(\nu) \leq_{\mathrm{st}} T_{l,0}(\nu) \leq_{\mathrm{st}} T_{L,0}(\nu)
$
(the two terms in the lower bound being dependent). It follows from Theorem~\ref{thm:expo} that $T_{L,0}(\nu)/ \E T_{L,0}(\nu) \cd \rmexp(1)$ as $\ninf$. In order to complete the proof, we can then use Lemma~\ref{lem:asymbounds}, taking $U(\nu)=T_{L,0}(\nu)$, $V(\nu)=T_{L,l}(\nu)$ and $W(\nu)=0$. The condition which needs to be checked is $\limnu \E V(\nu)/\E U(\nu) = 0$, which follows directly from Proposition~\ref{prop:meansc}.
\end{proof}

\subsection{More general coefficients and applications}
\label{subsec:gencoef}
We can extend our analysis to more general activation and deactivation dynamics inside a single branch, described by
\eqan{
&q(l, l+1) = a_l f(\nu), \quad l = 1, \dots, L-1,\nonumber \\
&q(l, l-1) = d_l, \quad l = 2, \dots, L, \nonumber
}
where $a_l,d_l$ are positive real coefficients. Specifically, Proposition~\ref{prop:meansc} can be generalized to the following result. For $L \geq l_1 > l_2 \geq 0$,
\eqn{\label{eq:genrates}
\E T_{l_1,l_2}(\nu) \sim \frac{1}{d_{l_2+1}} \Big( \prod_{i=l_2+1}^{L-1} \frac{a_i}{d_{i+1}} \Big) f(\nu)^{L-l_2-1}, \quad  \text{ as } \ninf.
}
Also Lemma~\ref{lem:asymptotic1} and thus Proposition~\ref{prop:ael0} can be shown to hold for these more general rates (see~\ref{ap1}).

These results for general coefficients have some interesting applications, beyond the model considered in this paper. One example is the continuous-time Markov process $\{M_t\}_{t \geq 0}$ on $S=\{0,1,\dots, c\}$, describing the number of busy servers at time $t$. Suppose that the service rate of each server is $1$ and the arrival rate is $\nu$ which will grow large in a heavy-traffic regime. The escape time $T_{s,0}(\nu)$, choosing $a_n=1$ and $d_n=n$, $n=1,\dots,c$, then describes the time it takes for this system to \textit{drain} (i.e.~to have all the servers idle) when starting with $s \geq 1$ busy servers. Then~\eqref{eq:genrates} gives $\E T_{s,0}(\nu) \sim \nu^{c-1} / c!$ as $\ninf$, which does not depend on the starting state $s \geq 1$. The scaled drain time obeys $ \frac{T_{s,0}(\nu)}{\E T_{s,0}(\nu)} \cd \rmexp(1)$ as $\ninf$.

\section{Proofs of Theorems~\ref{thm:thm1} and~\ref{thm:thm2}}
\label{sec5}
In this section we investigate the asymptotic behavior of the transition time $T_{(k_1,l_1), (k_2,l_2)}(\nu)$ as $\ninf$ for any pair of states $(k_1,l_1)$ and $(k_2,l_2)$, with $k_1\neq k_2$.
In Subsection~\ref{sec51} we provide a stochastic representation of the transition time, which we use to derive the asymptotic mean transition time in Subsection~\ref{sec52} leading to Theorem~\ref{thm:thm1}. In Subsection~\ref{sec53} we will obtain the asymptotic distribution of the scaled transition time leading to Theorem~\ref{thm:thm2}. In Subsection~\ref{sec54} we consider in detail the random variable $W$ that occurs in Theorem~\ref{thm:thm2}. We give an overview of all possible forms of asymptotic behavior and the conditions under which they occur in Subsection~\ref{sec55}.

\subsection{Stochastic representation of the transition time}
\label{sec51}
Consider the evolution of the process as it makes a transition from a state $(k_1,l_1)$ to a state $(k_2,l_2)$ and defining the following random variables:
\begin{itemize}
\item $T_{(k_1, 1), 0}^{(0)}(\nu)$: time to reach state~$0$ after state $(k_1, 1)$ is visited for the first time;
\item $N_k (\nu)$: number of times the process makes a transition $0 \to (k, 1)$, $k \neq k_2$, before the first transition $0 \to (k_2, 1)$ occurs;
\item $\hat{T}_{0, (k, 1)}^{(i)}(\nu)$: time spent in state~0 before the $i$-th transition to state $(k, 1)$, $k \neq k_2$, $i = 1, \dots, N_k(\nu)$;
\item $\hat{T}_{0, (k_2, 1)}(\nu)$: time spent in state~$0$ before the first transition to state $(k_2, 1)$;
\item $T_{(k, 1), 0}^{(i)}(\nu)$: time to return to state~$0$ after the $i$-th transition to state $(k, 1)$, $k \neq k_2$, $i = 1, \dots, N_k(\nu)$.
\end{itemize}
With the above definitions, it is readily seen that the following stochastic representation holds.
\begin{prop}
The transition time $\TT$ can be represented as
\eqn{ \label{eq:kpartiterepr}
\TT \ed T_{(k_1, l_1), (k_1, 1)} + T_{(k_1, 1), 0}^{(0)} +  \sum_{k \neq k_2} \sum_{i = 1}^{N_k} \left(\hat{T}_{0, (k, 1)}^{(i)} + T_{(k, 1), 0}^{(i)}\right) + \hat{T}_{0, (k_2, 1)} + T_{(k_2, 1), (k_2, l_2)},
}
where the dependence on the parameter~$\nu$ is suppressed for compactness and all the random variables representing time durations are mutually independent as well as independent of the random variables $N_k(\nu)$, $k \neq k_2$.
\end{prop}
Denote $F(\nu) = \sum_{k = 1}^{K} L_k f_k(\nu)$. The random variables $T_{(k,1), 0}^{(i)}$ are i.i.d.~copies of $T_{(k,1), 0}$, $i = 1, \dots, N_k(\nu)$, $k \neq k_2$, while the random variables $\hat{T}_{0, (k_2, 1)}$ and $\hat{T}_{0, (k, 1)}^{(i)}$, $k \neq k_2$, $i = 1, \dots, N_k$, are i.i.d.~copies of $T_0 \ed \rmexp(F(\nu))$, which is the residence time in state~0. Write $X \ed \mathrm{Geo}(p)$ when $X$ is a random variable with geometric distribution~$\pr{X=n}=p(1-p)^n$, $n \in \N \cup \{0\}$. Define the random variable $N (\nu):= \sum_{k \neq k_2} N_k (\nu)$, counting the total number of entrances in branches other than $k_2$ before hitting the target branch $\B_{k_2}$. For all $k = 1, \dots, K$, denote $p_k(\nu) := L_k f_k(\nu) / F(\nu)$. Obviously,
\[
N(\nu) \ed \mathrm{Geo}(p_{k_2}(\nu)),
\]
while the marginal distribution of $N_k(\nu)$ is $\mathrm{Geo} (\frac{p_{k_2}(\nu) }{ p_{k_2}(\nu)+p_k(\nu)})$.

We want to distinguish the branches that significantly affect the dynamics of the process (and hence the transition time) from those that do not. The quantity $ \E N_k(\nu) \cdot \E T_{(k, 1), 0}(\nu)$, for $k\neq k_2$, is the mean time that the process spends in branch $\B_k$ along the transition $(0,0)\to(k_2,l_2)$.
Note that Proposition~\ref{prop:meansc} gives
\eqn{\label{eq:met}
\E T_{(k,1), 0}(\nu) \sim \frac{1}{L_k} f_k(\nu)^{L_k-1}, \quad  \text{ as } \ninf,
}
and that
\eqn{\label{eq:enk}
\E N_k(\nu) = \frac{p_k(\nu) }{ p_{k_2}(\nu) }= \frac{L_k f_k(\nu) }{ L_{k_2} f_{k_2}(\nu)}.
}
Therefore
\[
\frac{\E N_k(\nu) \cdot \E T_{(k, 1), 0}(\nu)}{\E N_j(\nu) \cdot \E T_{(j, 1), 0}(\nu)} \sim \frac{f_k(\nu)^{L_k}}{ f_j(\nu)^{L_j}}, \quad  \text{ as } \ninf,
\]
which shows that indeed only the visits to dominant branches asymptotically contribute to the mean transition time.

\subsection{Asymptotic mean transition time}
\label{sec52}
We present here the proof of Theorem~\ref{thm:thm1}. Consider the stochastic representation~\eqref{eq:kpartiterepr} of the transition time $T_{(k_1, l_1), (k_2, l_2)}(\nu)$. Proposition~\ref{prop:meansc} implies that
\[
\E T_{(k_1,l_1),(k_1,1)} (\nu) \sim \frac{1}{L_{k_1} (L_{k_1}-1)} f_{k_1}(\nu)^{L_{k_1}-2}, \quad \text{ as } \ninf,
\]
and that
\[
\E T_{(k_1,1),0} (\nu) \sim \frac{1}{L_{k_1}} f_{k_1}(\nu)^{L_{k_1}-1}\quad \text{ as } \ninf.
\]
Hence $\E T_{(k_1,l_1),(k_1,1)(\nu)}=o\left(\E T_{(k_1,1),0} (\nu) \right)$ as $\ninf$. Moreover $\E \hat{T}_{0, (k, 1)}(\nu)=o(1)$. Lemma~\ref{lem:fifthterm} below implies that $\E T_{(k_2, 1), (k_2, l_2)}(\nu) = o\left(\E \TT(\nu)\right)$. The asymptotic relation~\eqref{eq:agr} then follows using the definition of $K_*$, the asymptotic estimate~\eqref{eq:met} and the identity~\eqref{eq:enk}.

The following lemma guarantees that once the process has entered the target branch $\B_{k_2}$, even if it may exit from it, the mean time it takes to reach the target state $(k_2,l_2)$ is negligible with respect to the mean overall transition time.

\begin{lem}\label{lem:fifthterm}
\[\E T_{(k_2, 1), (k_2, l_2)}(\nu) = o\left(\E \TT (\nu) \right), \quad  \text{ as } \ninf.\]
\end{lem}

\begin{proof} Consider the event
\eqan{
\CE(\nu)&=\{\text{the first } l_2-1 \text{ transitions are all towards the state }(k_2,l_2)\} \nonumber \\
&=\bigcap_{i=1}^{l_2-1} \{\text{the } i \text{-th transition is from }(k_2,i) \text{ to } (k_2,i+1)\}. \nonumber
}
Exploiting the fact that all these events are independent, we can compute
\eqan{
\pr{\CE(\nu)}&= \prod_{i=1}^{l_2-1} \pr{ \text{the } i \text{-th transition is from }(k_2,i) \text{ to } (k_2,i+1)} \nonumber\\
&=\frac{(L_{k_2}-1)f_{k_2}(\nu)}{(L_{k_2}-1)f_{k_2}(\nu) +1}\cdot \frac{(L_{k_2}-2)f_{k_2}(\nu)}{(L_{k_2}-2)f_{k_2}(\nu) +2} \cdot \dots \cdot \frac{(L_{k_2}-l_2+1)f_{k_2}(\nu) }{(L_{k_2}-l_2+1)f_{k_2}(\nu) + (l_2+1)}, \nonumber
}
and clearly $\limnu \pr{\CE(\nu)} =  1$. We have that
\[
\E \{T_{(k_2, 1), (k_2, l_2)}(\nu)~|~ \CE(\nu)\} = \sum_{m=1}^{l_2-1} \frac{1}{(L_2-m)f_{k_2}(\nu) + m}=:g(\nu),
\]
where $g(\nu)\downarrow 0$ as $\ninf$. Consider the events $\CE^c_n(\nu)=\{ \text{the first transition towards state }0\text{ is the } n\text{-th one}\}$, for $n=1,\dots,l_2-1$. Note that the event $\CE^c(\nu)$ can be decomposed as $\CE^c(\nu)=\bigcup_{n=1}^{l_2-1} \CE_n^c(\nu)$.
Using the events $\CE(\nu)$ and $\CE^c_n(\nu)$, we can write
\eqn{ \label{eq:reprA}
\E T_{(k_2, 1), (k_2, l_2)}(\nu)=\E \{T_{(k_2, 1), (k_2, l_2)} (\nu) ~|~ \CE(\nu) \} \pr{\CE (\nu)}+\sum_{n=1}^{l_2-1} \E \{ T_{(k_2, 1), (k_2, l_2)} (\nu) ~|~ \CE^c_n(\nu) \} \pr{\CE^c_n(\nu)}.
}
Since
\eqan{
\E \{T_{(k_2, 1), (k_2, l_2)}(\nu) ~|~ \CE^c_n(\nu) \} &\leq \E \{T_{(k_2, 1), (k_2, n-1)}(\nu) ~|~ \CE^c_n(\nu) \}+\E T_{(k_2, n-1), (k_2, l_2)}(\nu) \nonumber \\
&\leq \E \{T_{(k_2, 1), (k_2, n-1)}(\nu) ~|~ \CE(\nu) \}+\E T_{0, (k_2, l_2)}(\nu) \nonumber \\
&\leq \E \{T_{(k_2, 1), (k_2, l_2)}(\nu) ~|~ \CE(\nu) \}+\E T_{(k_1,l_1), (k_2, l_2)}(\nu), \nonumber
}
for $n=1, \dots, l_2-1$, it follows from~\eqref{eq:reprA} that
\[
\E T_{(k_2, 1), (k_2, l_2)}(\nu) \leq g(\nu)+\E T_{(k_1,l_1), (k_2, l_2)}(\nu) \pr{\CE^c (\nu)}.
\]
We divide both sides by $\E T_{(k_1,l_1), (k_2, l_2)}(\nu)$, which is greater than~1 for $\nu$ sufficiently large, thanks to~\eqref{eq:sdp}. Since $g(\nu)$ and $\pr{\CE^c (\nu)}$ are both $o(1)$, the proof of the lemma is complete.
\end{proof}

\subsection{Asymptotic distribution of the transition time}
\label{sec53}
We now turn to the proof of Theorem~\ref{thm:thm2}. It is clear that only the dominant branches which asymptotically contribute to the expected magnitude of the transition time will play a role, possibly along with the escape time from the initial branch. As we will show, the various dominant branches may play different roles, depending on whether the expected number of visits during the transition time is zero, $O(1)$ or infinite in the limit as $\ninf$. We introduce
\eqn{\label{eq:ab}
A(\nu):=T_{(k_1,1),0}(\nu) \quad \text{ and } \quad B(\nu):=\sum_{k \in K_*} \sum_{i=1}^{N_k(\nu)}T_{(k,1),0}^{(i)}(\nu),
}
whose means correspond to the two terms at the right-hand side of~\eqref{eq:agr}. From the definition~\eqref{eq:alpha} of the coefficient $\alpha$ it follows that
\[
\alpha=\limnu \frac{\E A (\nu)}{\E A(\nu) + \E B(\nu)}.
\]
When $\alpha=0$ the term $A(\nu)$ becomes asymptotically negligible compared to $B(\nu)$, while the opposite holds when $\alpha=1$. Proposition~\ref{prop:ael0} already describes the asymptotic behavior of $A(\nu)$ after scaling. We need to understand the asymptotic behavior of $B(\nu)$ and for this purpose, it will be convenient to use a slightly different representation for it.

Define $p_*(\nu):=\sum_{k \in K_*} p_k(\nu)$ and $\hat{p}(\nu):=\frac{p_*(\nu) }{ p_{k_2}(\nu) +p_*(\nu)}$. Introduce the random variable $N_*(\nu):=\mathrm{Geo}(1-\hat{p}(\nu))$, which represents the number of visits to the dominant branches, before entering the target branch $\B_{k_2}$. Introduce the sequence $(\tau^{(i)}(\nu))_{i \geq 1}$ of i.i.d.~random variables, $\tau^{(i)}(\nu) \ed \tau(\nu)$, where $\tau(\nu) \ed T_{(k,1),0}(\nu)$ with probability $p_k(\nu)/p_*(\nu)$ for every $k \in K_*$. Then
\eqn{\label{eq:reprtau}
B(\nu) \ed \sum_{i=1}^{N_*(\nu)} \tau^{(i)}(\nu).
}
For $k \in K_*$ we define
\eqn{\label{eq:defbk}
\b_k:= \limnu \frac{L_k f_k(\nu)}{L_{k_2} f_{k_2}(\nu)}.
}
In view of~\eqref{eq:sd}, $\b_k$ may be interpreted as the stationary ratio between the number of visits to branch $k$ and to branch $k_2$ as $\ninf$. Starting from branch $k_1 \neq k_2$, $\b_k$ also represents the asymptotic mean number of visits to branch $\B_k$ before the first entrance in branch $\B_{k_2}$ as $\ninf$. To avoid technicalities, we henceforth assume that all the parameters $\b_k$ are well defined. Moreover, we introduce the parameter $\b:=\sum_{k \in K_*} \beta_k$, which is the asymptotic mean number of visits to dominant branches before hitting $\B_{k_2}$ as $\ninf$, i.e. $\b=\limnu \E N_*(\nu)$.

Based on the definition of the parameter $\b_k$ in~\eqref{eq:defbk}, we partition the index set $K_*$ of the dominant branches into three subsets, namely
\[
K_* = \NA \cup \A \cup \SA,
\]
using the following rule:
\begin{itemize}
\item[(i)] $k\in \NA$ if $\b_k=0$;
\item[(ii)]$k \in \A$ if $\b_k\in \R_+$;
\item[(iii)] $k\in \SA$ if $\b_k=\infty$.
\end{itemize}
The branches in $\NA$, $\A$ and $\SA$ will be called \textit{non-attracting}, \textit{attracting} and \textit{strongly attracting}, respectively. Define moreover the coefficients $\g_{\NA}:=\sum_{k \in \NA}\g_k$, $\g_{\A}:=\sum_{k \in \A}\g_k$ and $\g_{\SA}:=\sum_{k \in \SA}\g_k$, with the parameters $\g_k$ as defined in~\eqref{eq:defgk}.

We are now ready to present the proof of Theorem~\ref{thm:thm2}. Specifically, we prove that if $k_1\neq k_2$, $1\leq l_1 \leq L_{k_1}$ and $1\leq l_2 \leq L_{k_2}$, then
\[
\frac{\TT(\nu)}{ \E \TT (\nu)} \cd \alpha Y +(1-\alpha) W, \quad  \text{ as } \ninf,
\]
where $\alpha$ is the constant defined in~\eqref{eq:alpha}, $Y$ is an exponential random variable with unit mean and $W$ is a random variable independent of $Y$, with Laplace transform
\eqn{\label{eq:LTW}
\LT_W (s)=\frac{1}{\displaystyle 1+\sum_{k \in \A} \frac{\g_k s}{1 + \g_k s / \beta_k} + s \g_{\SA}}.
}

The crucial idea of the proof is to use Lemma~\ref{lem:asymbounds} with the dominant term $U(\nu)$ defined as the sum of the two random variables introduced in~\eqref{eq:ab}, i.e.~$U(\nu):=A(\nu)+B(\nu)$. Theorem~\ref{thm:thm1} implies that $\E \TT (\nu) \sim \E U(\nu)$ as $\ninf$ and its proof shows that all the other terms present in the stochastic representation~\eqref{eq:kpartiterepr} are negligible compared to $U(\nu)$. Note that
\[
\frac{U(\nu)}{\E U(\nu)} = \frac{A(\nu)}{\E U(\nu)} + \frac{B(\nu)}{\E U(\nu)} = \frac{\E A(\nu)}{\E U(\nu)} \frac{A(\nu)}{\E A(\nu)} + \frac{\E B(\nu)}{\E U(\nu)} \frac{B(\nu)}{\E B(\nu)}.
\]
Recall that $A(\nu)$ and $B(\nu)$ are independent by construction. If we knew that there exist two random variables $Y$ and $W$ such that $A(\nu) /\E A(\nu) \cd Y$ and $B(\nu) /\E B(\nu) \cd W$ as $\ninf$, then
\[
\frac{U(\nu)}{\E U(\nu)} \cd \alpha Y + (1-\alpha) W, \quad \text{ as } \ninf,
\]
and Lemma~\ref{lem:asymbounds} would imply that
\[
\frac{\TT (\nu)}{\E \TT (\nu)} \cd \alpha Y + (1-\alpha) W, \quad \text{ as } \ninf.
\]
Proposition~\ref{prop:ael0} immediately gives that
\[
\frac{A(\nu)}{\E A(\nu)} \cd Y, \quad \text{ as } \ninf,
\]
where $Y$ is an exponential random variable with mean one.

Thus it remains to establish that the random variable $B(\nu)/\E B(\nu)$ converges to $W$ in distribution. From the definition of $B(\nu)$ and~\eqref{eq:reprtau}, it follows that $\E B(\nu) =\E N_*(\nu) \E \tau(\nu)$ and that
\eqn{\label{eq:LTB}
\LT_{B(\nu)/\E B(\nu)}(s)=G_{N_*(\nu)}\left(\LT_{\tau(\nu)/\E B(\nu)}(s)\right)=G_{N_*(\nu)}\left(\LT_{\tau(\nu)/\E \tau(\nu)}(s/\E N_*(\nu)\right),
}
where
\eqn{\label{eq:genfunN}
G_{N_*(\nu)}(z)=\E (z^{N_*(\nu)})=\frac{1}{1+(1-z)\E N_*(\nu)}.
}
We need to understand the asymptotic behavior of the random variable $\tau(\nu)/\E \tau(\nu)$. Let $T_k(\nu)=T_{(k,1),0}(\nu)$.
Then $\LT_{\tau(\nu)}(s)=\sum_{k \in K_*} \frac{p_k(\nu)}{p_*(\nu)} \LT_{T_k(\nu)}(s)$ and hence
\eqan{
\LT_{\tau(\nu)/\E \tau(\nu)}(s/\E N_*(\nu))
&= \LT_{\tau(\nu)} \Big(\frac{s}{\E N_*(\nu) \E \tau(\nu)}\Big) \nonumber \\
&= \sum_{k \in K^*} \frac{p_k(\nu)}{p_*(\nu)} \LT_{T_k(\nu)} \Big(\frac{s}{\E N_*(\nu) \E \tau(\nu)}\Big) \nonumber \\
&= \sum_{k \in K^*} \frac{\E N_k(\nu)}{\E N_*(\nu)} \LT_{T_k(\nu)/ \E T_k(\nu)} \Big(\frac{s \E T_k(\nu)}{\E N_*(\nu) \E \tau(\nu)}\Big). \nonumber
}
For $k \in K_*$, define
\eqn{\label{eq:gnu}
h_k(\nu):=\frac{\E T_k(\nu)}{\E N_*(\nu) \E \tau(\nu)},
}
and note that $\limnu h_k(\nu)= \g_k / \beta_k$. Indeed,
\[
\limnu h_k(\nu) = \limnu \frac{\E T_k(\nu)}{\E N_*(\nu) \E \tau(\nu)}= \limnu \frac{\E N_k(\nu) \E T_k(\nu)}{\E N_*(\nu) \E \tau(\nu)} \frac{1}{\E N_k(\nu)}= \g_k \left (\limnu \E N_k(\nu)\right)^{-1}.
\]
Combining~\eqref{eq:LTB}-\eqref{eq:gnu} yields
\eqan{ \label{eq:limBf}
\LT_{B(\nu)/\E B(\nu)}(s)
&= \left [ 1+\Big(1- \LT_{\tau(\nu)/\E \tau(\nu)}(s/\E N_*(\nu)) \Big) \E N_*(\nu) \right]^{-1} \nonumber \\
&= \left [ 1+\Big(1- \sum_{k \in K^*} \frac{\E N_k(\nu)}{\E N_*(\nu)} \LT_{T_k(\nu)/ \E T_k(\nu)} \Big (\frac{s \E T_k(\nu)}{\E N_*(\nu) \E \tau(\nu)}\Big) \Big) \E N_*(\nu) \right]^{-1} \nonumber \\
&= \left [ 1+\Big(\E N_*(\nu) - \sum_{k \in K^*} \E N_k(\nu) \LT_{T_k(\nu)/ \E T_k(\nu)} (s h_k(\nu)) \Big) \right]^{-1} \nonumber \\
&= \left [ 1+ \sum_{k \in K^*} \E N_k(\nu)\left  ( 1- \LT_{T_k(\nu)/ \E T_k(\nu)} (s h_k(\nu))\right ) \right]^{-1}.
}

In order to obtain an explicit expression for $\LT_{B(\nu)/\E B(\nu)}(s)$ as $\ninf$, we need the following technical lemma, which is proved in~\ref{ap3}.
\begin{lem} \label{lem:lt}
\begin{itemize}
\item[{\rm (a)}] If $k \in \SA$, then
\[
\limnu \E N_k(\nu)\left (1- \LT_{T_k(\nu)/ \E T_k(\nu)} (s h_k(\nu)) \right ) = \g_k s.
\]
\item[{\rm (b)}]  If $k \in \A$, then
\[
\limnu \E N_k(\nu)\left (1- \LT_{T_k(\nu)/ \E T_k(\nu)} (s h_k(\nu)) \right ) = \frac{\g_k s}{1+\g_k s / \b_k}.
\]
\item[{\rm (c)}]  If $k \in \NA$, then
\[
\limnu \E N_k(\nu)\left (1- \LT_{T_k(\nu)/ \E T_k(\nu)} (s h_k(\nu)) \right ) = 0.
\]
\end{itemize}
\end{lem}

From Lemma~\ref{lem:lt} and~\eqref{eq:limBf} it follows that
\[
\LT_W(s) = \limnu \LT_{B(\nu)/\E B(\nu)}(s)=\left [ 1+\sum_{k \in \A} \frac{\g_k s}{1 + \g_k s / \beta_k} + \sum_{k \in \SA} \g_k s \right ]^{-1}.
\]
The independence of $Y$ and $W$ easily follows from the independence of the corresponding terms in the stochastic representation~\eqref{eq:kpartiterepr}.

\subsection{The random variable $W$: properties and interpretation}
\label{sec54}
The random variable $W$ is defined by its Laplace transform, see~\eqref{eq:LTW}. We remark that the shape of the distribution $W$ is fully determined by the branches in $\A$ and $\SA$, independently of the branches in $\NA$. Indeed the random variable $W$ can be represented as
\[
W \ed (1-\g_{\NA}) \overline{W},
\]
where $\overline{W}$ is a unit-mean random variable that in no way depends on the parameters of the branches in the set $\NA$. On the other hand, the presence of the factor $(1-\g_{\NA})$ reflects the fact that the branches in $\NA$ do affect the mean of the asymptotic scaled transition time: indeed convergence of the first moments holds if and only if $\alpha=1$ or $\NA=\emptyset$. Indeed,
\[
\alpha \E Y +(1-\alpha) \E W = \alpha + (1-\alpha)(1-\g_{\NA}),
\]
and, if $\NA\neq\emptyset$, then $\g_{\NA} >0$ and so $\alpha \, \E Y +(1-\alpha) \, \E W<1$ when $\alpha \neq 1$.

Whenever either $\A$ or $\SA$ is empty, the distribution of $W$ is known explicitly, cf.~Table~\ref{tab:overview}. However, also in the scenario where both $\A$ and $\SA$ are non-empty, it is still possible to give an interpretation of the distribution of $W$.

If $\A\neq \emptyset$, define $m:=|\A|$ and label the branches belonging to $\A$ as $1, 2, \dots, m$. Let $\b_{\A}:=\sum_{k=1}^m \beta_k \in (0,\infty)$ be the asymptotic mean number of visits to attracting branches as $\ninf$. Consider a hyper-exponentially distributed random variable $H$ with rates $\b_k/\g_k$ and probabilities $\b_k/\b_{\A}$, $k=1,\dots,m$, whose Laplace transform is 
\[
\LL_{H}(s)=\sum_{k=1}^m \frac{\b_k}{\b_{\A}} \frac{\b_k/\g_k}{\b_k/\g_k+s}.
\]

Furthermore consider a marked Poisson process with rate $\lambda=\b_{\A} / \g_{\SA}$ and i.i.d.~marks distributed according to $H$. The random variable $W$ in Equation~\eqref{eq:LTW} corresponds to the sum of a random time $\mathcal T$, with $\mathcal T$ exponentially distributed with mean $1/\mu = \g_{\SA}$, and the total size $\mathcal W(\mathcal T)$ of the marks associated with all the events in interval $[0,\mathcal T]$. Indeed

\begin{align*}
\displaystyle \LT_{\mathcal T + \mathcal W (\mathcal T)}(s) &= \displaystyle \int_{t =0}^\infty e^{-s t} e^{\lambda t \, \left (\sum_{k=1}^m \frac{\b_k}{\b_{\A}} \frac{\b_k/\g_k}{\b_k/\g_k+s} - 1\right )} \mu e^{-\mu t} d t  = \left [ 1+\frac{\lambda}{\mu} \left (\sum_{k=1}^m \frac{\b_k}{\b_{\A}} \frac{\b_k/\g_k}{\b_k/\g_k+s} \right )+ \frac{s}{\mu} \right ]^{-1} \nonumber \\
& = \displaystyle \left [ 1+ \beta_{\A} \left (\sum_{k=1}^m \frac{\beta_k}{\beta_{\A}} \frac{s}{\beta_k/\g_k +s} \right )+ s \g_{\SA} \right ]^{-1} = \left [ 1+\sum_{k \in \A} \frac{\g_k s}{1 + \g_k s / \beta_k} + s \g_{\SA} \right ]^{-1} .
\end{align*}

The stochastic equality $W = \mathcal T + \mathcal W (\mathcal T)$ may be interpreted as follows. Define $p_{\A}:= \sum_{k \in \A} p_k(\nu)$ and $p_{\SA}:= \sum_{k \in \SA} p_k(\nu)$. The total number of visits during the transition time to the branches in $\SA$ is geometrically distributed with parameter $p_{k_2} / p_{\SA}$. Since the durations of these visits are independent and each relatively short compared to the transition time, the total normalized amount of time spent in the branches in $\SA$ is exponentially distributed in the limit as $\ninf$ with mean $\g_{\SA}$. The visits to the branches in $\SA$ are interspersed with visits to the branches in $\A$. The number of visits to branches in $\SA$ between two consecutive visits to branches in $\A$ is geometrically distributed with parameter $p_{\A} / p_{\SA}$. The normalized durations of the visits to the branches in $\A$ have the hyper-exponential distribution $H$ as specified above. By similar arguments as mentioned above, the normalized amounts of time between these visits are independent and exponentially distributed in the limit as $\ninf$ with mean $\g_{\SA} \cdot p_{k_2} / p_{\A} = \g_{\SA} / \b_{\A}$. In other words, the visits to the branches in $\A$ occur as a Poisson process with rate $\lambda = \b_{\A} / \g_{\SA}$.

\subsection{An overview of the possible limiting distributions}
\label{sec55}
In this subsection we present an overview of all the possible limiting distributions of the scaled transition time by means of Table~\ref{tab:overview} and some simulation results to illustrate our findings. 

Denote by $(H_i)_{i \in \N}$ a sequence of i.i.d.~hyper-exponential random variables, $H_i \ed H$, while $\mathcal G$ is a geometric random variable $\mathcal G \ed \mathrm{Geo}\big (\frac{1}{1+\b_{\A}}\big )$, independent of all the other random variables.

\begin{table}[!hb]
\centering
{\renewcommand{\arraystretch}{1.5}
\begin{tabular}{ |c|c|c|l|c|}
\hline
$\alpha$ & $\A$ & $\SA$ & \hspace{3cm} Limiting distribution  & Scenario\\
\hline

\multirow{7}{*}{$0$}
 & $\emptyset$ & $\emptyset$ & $\d_0$ \quad (trivial r.v.~identical to $0$)  & 1a\\
 \cline{2-5}
 & non-empty & $\emptyset$ & $\displaystyle \sum_{i=1}^{\mathcal G} H_i(p_1,\dots,p_m,\lambda_1,\dots,\lambda_m)$ &1b\\
 & 		  &  & $\displaystyle \sum_{i=1}^{\mathcal G}  \rmexp_i(\lambda)$ \hspace{3.6cm} if $\b_k/\g_k = \lambda \quad \forall \, k \in \A$ &1b*\\
 \cline{2-5}
 & $\emptyset$ & non-empty & $\displaystyle \rmexp(1/\g_{\SA})$ &1c\\
 \cline{2-5}
 & non-empty & non-empty & $\displaystyle W$ &1d\\
\hline
\multirow{10}{*}{$(0,1)$}
 & $\emptyset$ & $\emptyset$ & $\displaystyle \rmexp(1/\alpha)$ &2a\\
 \cline{2-5}
 & &  & $\displaystyle \rmexp(1/\alpha) + \sum_{i=1}^{\mathcal G}  H_i \Big (\frac{\b_1}{\b_{\A}},\dots,\frac{\b_m}{\b_{\A}},\frac{\b_1}{(1-\alpha)\g_1},\dots,\frac{\b_m}{(1-\alpha)\g_m} \Big )$ &2b\\
 & non-empty & $\emptyset$ &  $\displaystyle \rmexp(1/\alpha)+ \sum_{i=1}^{\mathcal G}  \rmexp_i \Big (\frac{\lambda}{1-\alpha}\Big )$ \hspace{0.8cm} if $\b_k/\g_k = \lambda \quad \forall \, k \in \A$ &2b*\\
 &  &  &  $\displaystyle \rmexp\Big(\frac{1}{\alpha(1+\sum_{k=1}^m \b_k)}\Big)$ \hspace{1.1cm} if $\displaystyle \b_k/\g_k =\frac{1-\alpha}{\alpha} \quad \forall \, k \in \A$ &2b**\\
 &  &  &  $\displaystyle \rmexp(1)$ \hspace{2.2cm} if $\displaystyle \b_k/\g_k =\frac{1-\alpha}{\alpha}=\sum_{i=1}^m \b_k \quad \forall \, k \in \A$ &2b***\\
 \cline{2-5}
 & $\emptyset$ & non-empty & $\displaystyle \rmexp(1/\alpha) + \rmexp (1/ (1-\alpha)\g_{\SA})$ &2c\\
 &  &  & $\displaystyle \mathrm{Erlang}(2,1/\alpha)$ \hspace{3.8cm} if $\displaystyle  \alpha=\g_{\SA}/ (1+\g_{\SA})$ &2c*\\
 \cline{2-5}
 & non-empty & non-empty & $\displaystyle \rmexp(1/\alpha )+(1-\alpha) W$ &2d\\
\hline
$1$ & - & - & $\displaystyle \rmexp(1)$ &3\\
\hline
\end{tabular}
}
\caption{Overview of the possible asymptotic distributions of the scaled transition time.}
\label{tab:overview}
\end{table}

The case $\alpha=1$ always yields asymptotic exponentiality: this happens when the escape time from branch $\B_{k_1}$ dominates the total transition time. As soon as $\alpha \neq 1$, the set of dominant branches starts to play an important role. In particular, the shape of the asymptotic distribution depends only on the branches in the sets $\A$ and $\SA$ and changes substantially whenever one of these two subsets (or both) are empty. 

In the case $\alpha=0$ a diverse range of behaviors may occur, with asymptotic exponentiality only in a somewhat degenerate special case 1c. The behavior for $\alpha \in (0,1)$ is just a weighted combination of the extreme cases $\alpha=0$ and $\alpha=1$, as described in Theorem~\ref{thm:thm2}. It does not give rise to fundamentally different behavior, but interestingly enough, it does yield asymptotic exponentiality in some very special cases.

If all users have the same activation rate, no matter which component they belong to, then without loss of generality, we may assume $f_k(\nu)=\nu$, $k=1,\dots,K$. Under this homogeneity assumption, the sizes of components become crucial. Indeed, if one defines $L_*:=\max_{k \neq k_2} L_k$ to be the size of the largest component, then $K_*=\{k \neq k_2 : L_k=L_* \}$. In this case the orders of magnitude of the two dominant terms of the stochastic representation~\eqref{eq:kpartiterepr} are
\[
\E A(\nu) \sim \frac{\nu^{L_{k_1}-1}}{L_{k_1}} \quad \text{ and } \quad \E B(\nu) \sim \frac{|K_*| \nu^{L_*-1}}{L_{k_2}}, \quad \text{ as } \ninf,
\]
and hence for $1 \leq l_1 \leq L_{k_1}$, $1\leq l_2 \leq L_{k_2}$ and $k_1 \neq k_2$,
\[
\E T_{(k_1, l_1), (k_2, l_2)}(\nu) \sim \left(\frac{I_{\{k_1 \in K_*\}}}{L_*} + \frac{|K_*|}{L_{k_2}}\right) \nu^{L_* - 1}, \quad \text{ as } \ninf.
\]
Moreover, $\b_k/\g_k = (1-\alpha)/\alpha$ for every $k \in \A$ and thus only two possible scenarios can occur, namely 1b* and 2b*** (see \cite{ZBvL12}). The discriminating factor between these two scenarios is $\alpha$. If $k_1 \notin K_*$, then $\alpha=0$ and thus we are in scenario 1b*. If instead $k_1 \in K_*$, then $\alpha= L_{k_2}/ (|K_*| L_*)$, which means that scenario 2b*** occurs and hence asymptotic exponentiality arises.

To illustrate the range of possible limiting distributions, we present some simulation results. We consider the simplest system that is sufficiently rich to show the wide range of behaviors presented in Table~\ref{tab:overview}. Specifically, we consider a complete $3$-partite network, whose three components have sizes $L_1,L_2$ and $L_3$, and assume that the process starts in state $(1,L_1)$ and the target state is $(3, L_3)$. 

We use activation rates of the form $f_k(\nu)=\nu^{a_k}$. For compactness, we write $\underline{a}$ for $(a_1,a_2,a_3)$ and $\underline{L}$ for $(L_1,L_2,L_3)$. This choice for the activation rates allows to invert the Laplace transform of $W$ in all the cases and thus obtain a probability density function $f(x)$, which can be compared with the simulation data. 

All the simulations have been performed choosing the parameter $\nu=150$ and simulating the transition time for each network $20000$ times with a customized code written in the programming language C. The results are shown in Figures~\ref{fig:simnew1} and~\ref{fig:simnew2}.

\begin{figure}[!ht]
\centering
\subfigure[Case 1a: $\underline{L}=(3,4,6)$, $\underline{a}=(1,1,5/3)$]{\includegraphics{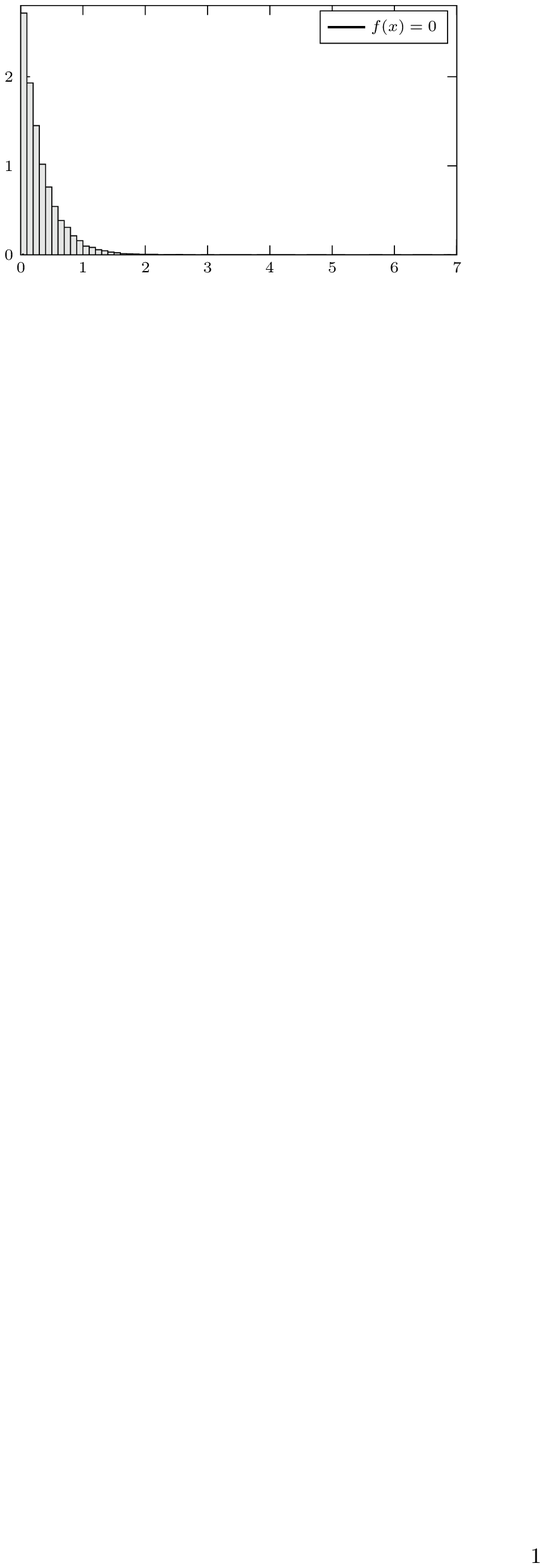}}
\hspace{0.5cm}
\subfigure[Case 1b: $\underline{L}=(3,5,5)$, $\underline{a}=(1/2,1/2,1/2)$]{\includegraphics{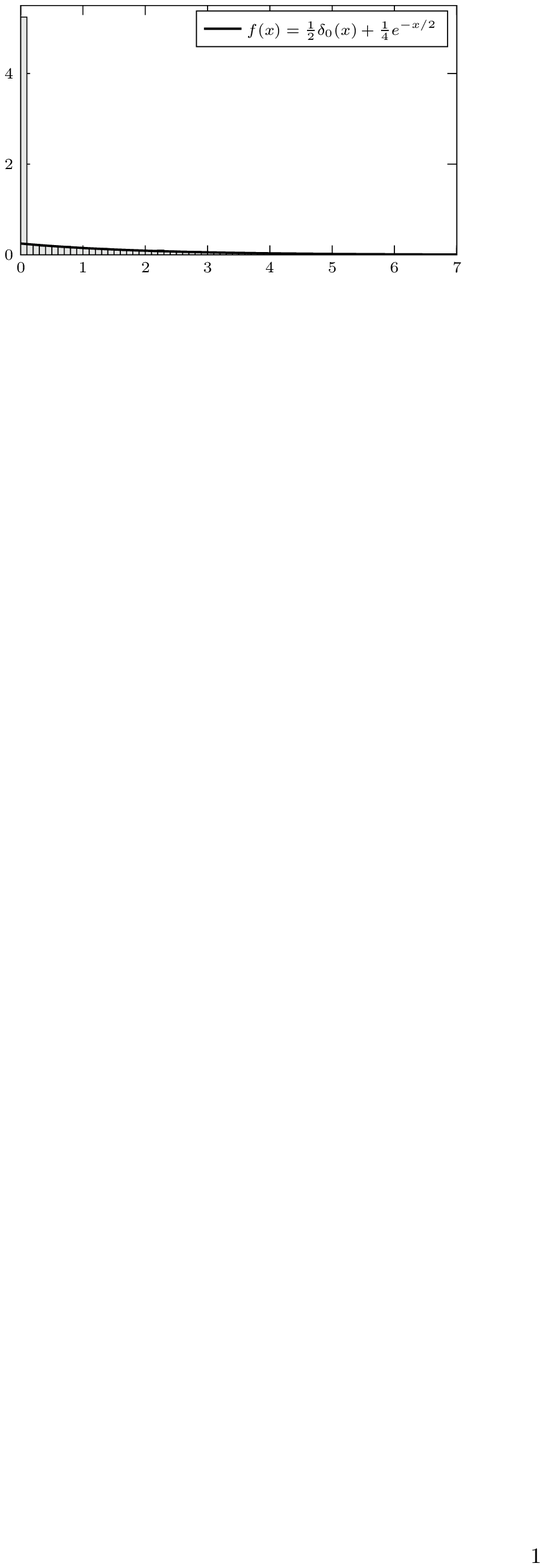}}
\\
\subfigure[Case 1c: $\underline{L}=(3,3,3)$, $\underline{a}=(4/5,3/5,3/5)$]{\includegraphics{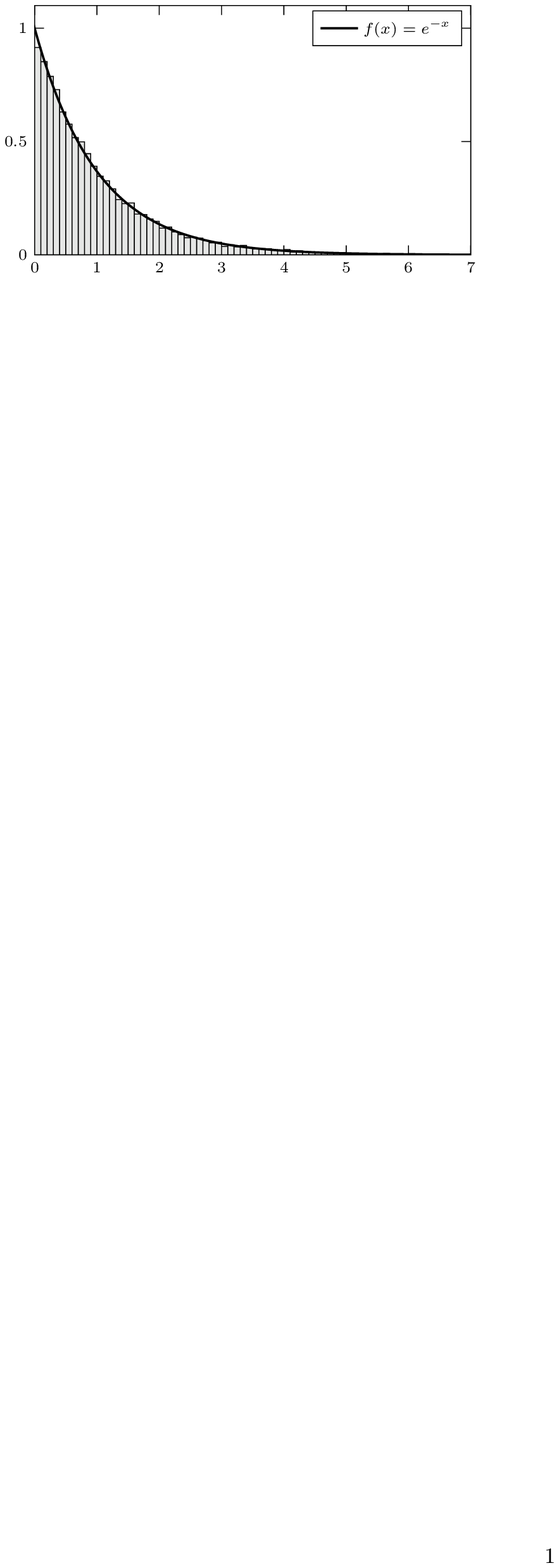}}
\hspace{0.5cm}
\subfigure[Case 1d: $\underline{L}=(3,4,6)$, $\underline{a}=(1,3/4,3/4)$]{\includegraphics{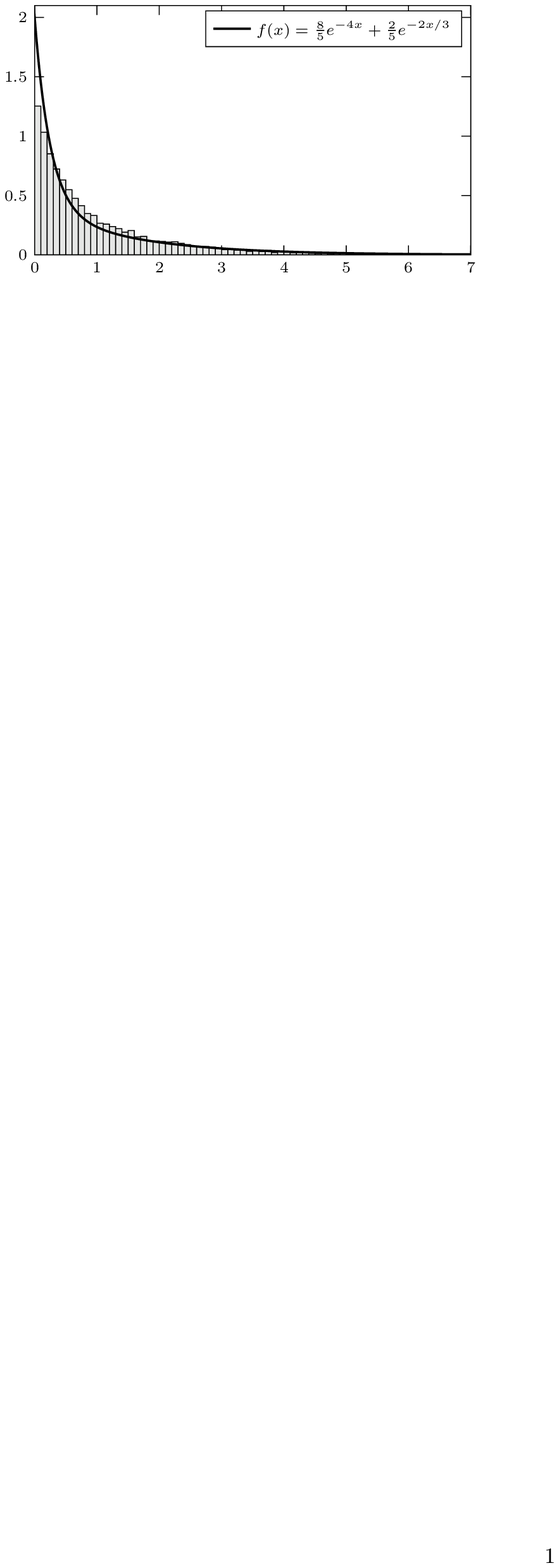}}
\caption{Plots of the empirical probability density function of the scaled transition times and the density $f(x)$ of $\alpha Y + (1-\alpha)W$.}
\label{fig:simnew1}
\end{figure}
\clearpage
\begin{figure}[!h]
\centering
\subfigure[Case 2a: $\underline{L}=(3,4,2)$, $\underline{a}=(9/10,9/10,9/5)$]{\includegraphics{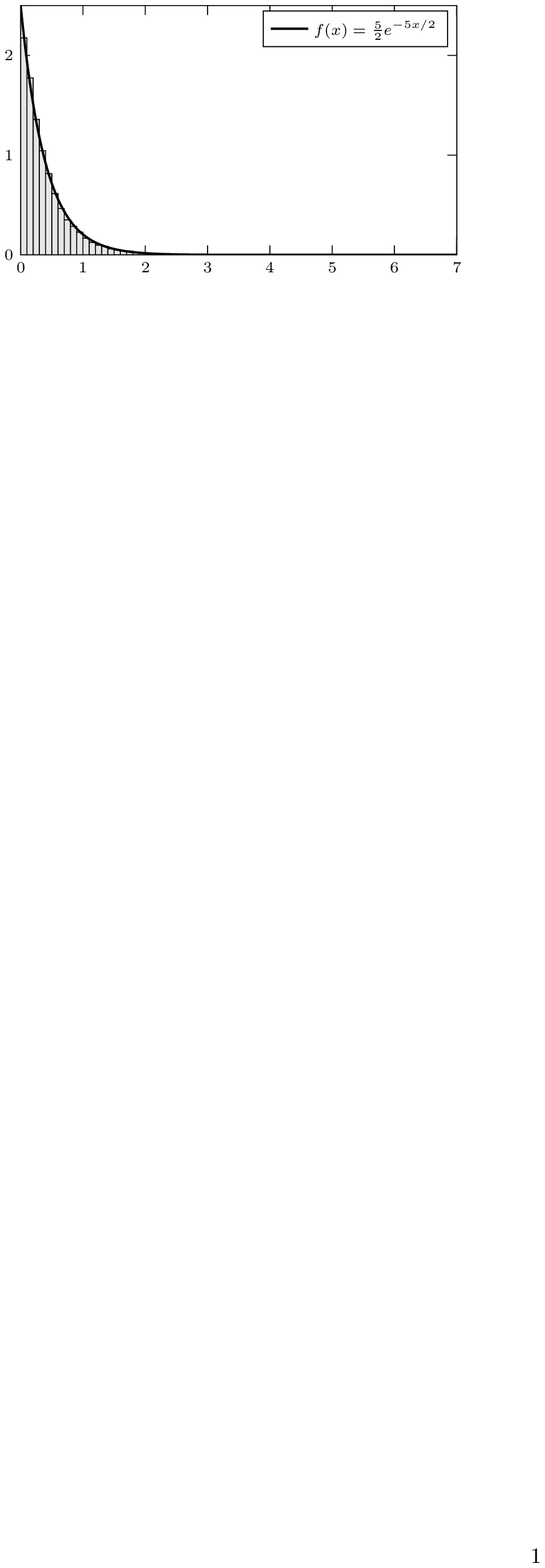}}
\hspace{0.5cm}
\subfigure[Case 2b*: $\underline{L}=(4,3,4)$, $\underline{a}=(2/3,1,1)$]{\includegraphics{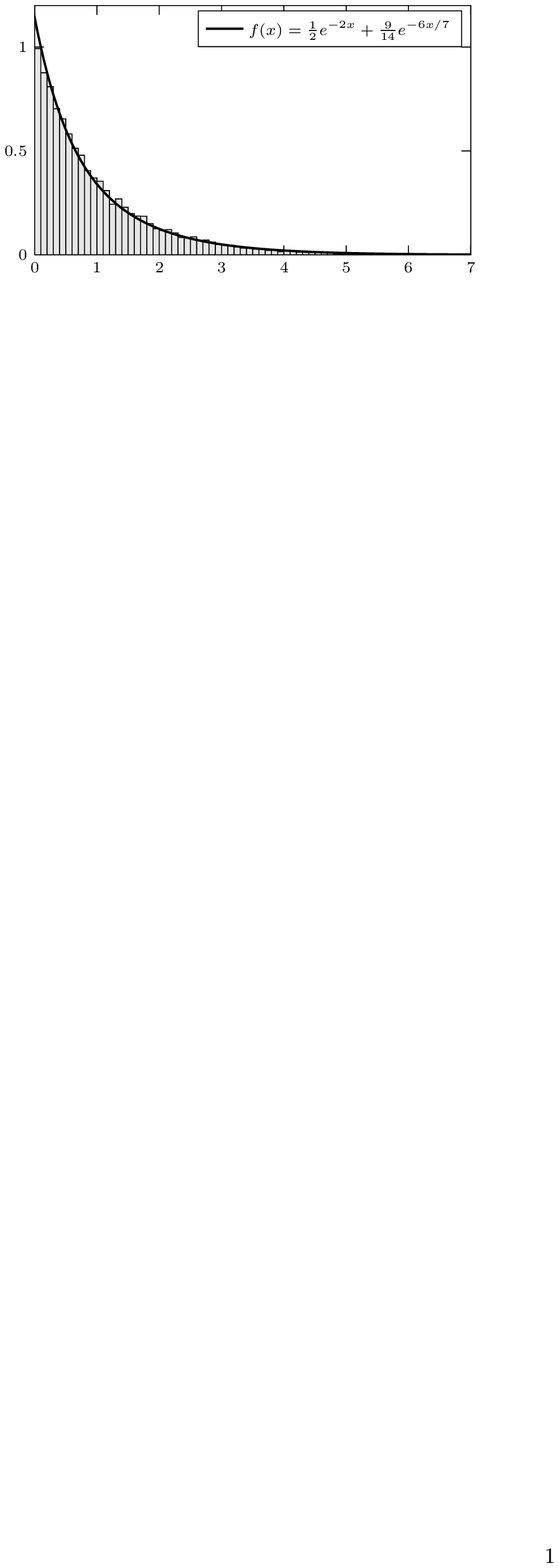}}
\\
\subfigure[Case 2b**: $\underline{L}=(2,4,5)$, $\underline{a}=(7/4,7/8,7/4)$]{\includegraphics{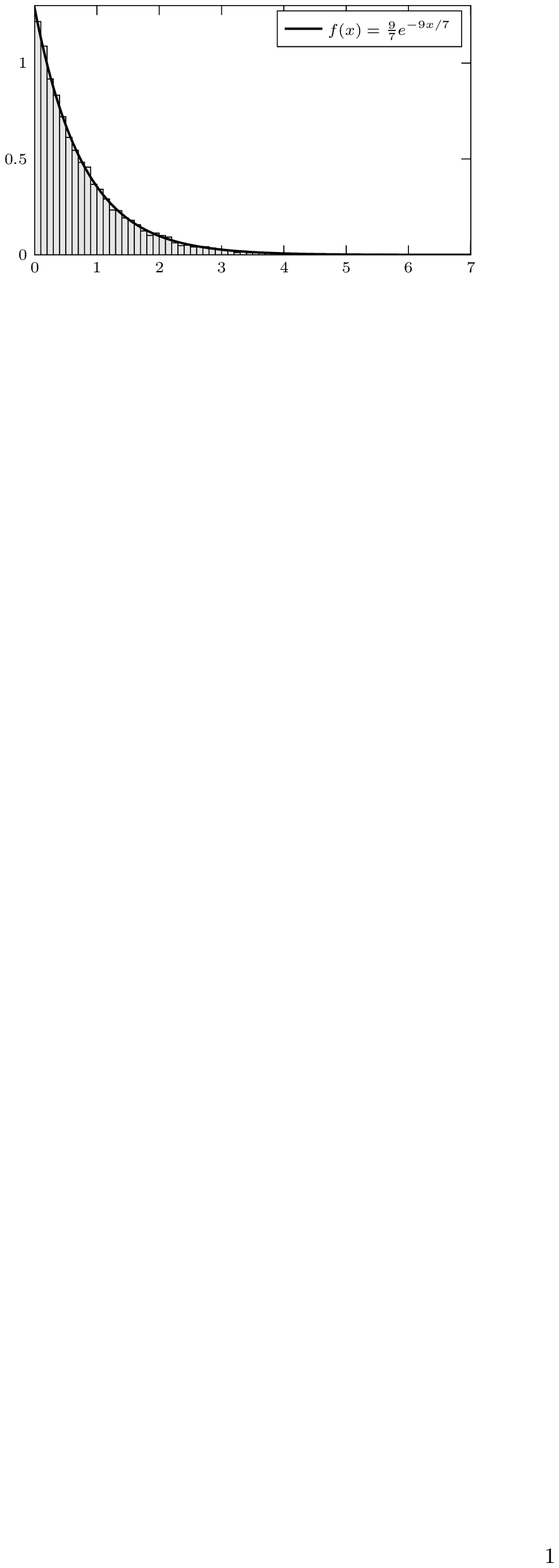}}
\hspace{0.5cm}
\subfigure[Case 2b***: $\underline{L}=(3,3,5)$, $\underline{a}=(7/8,7/8,7/8)$]{\includegraphics{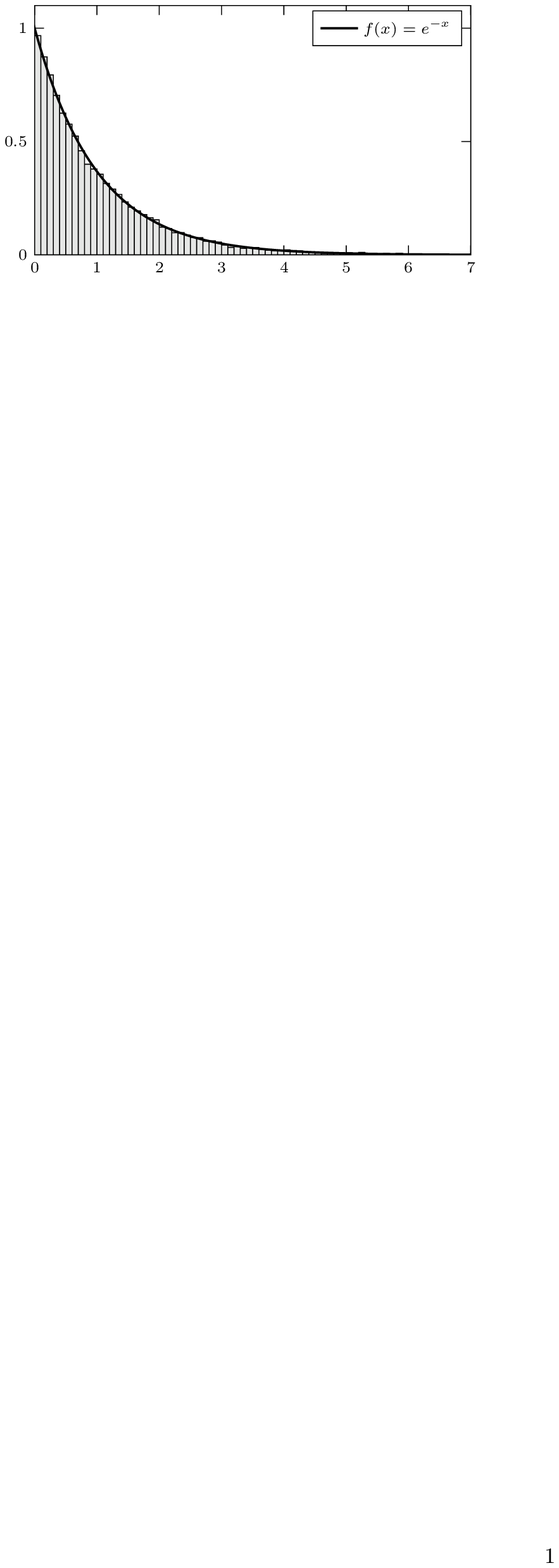}}
\\
\subfigure[Case 2c: $\underline{L}=(5,2,2)$, $\underline{a}=(4/9,4/3,8/9)$]{\includegraphics{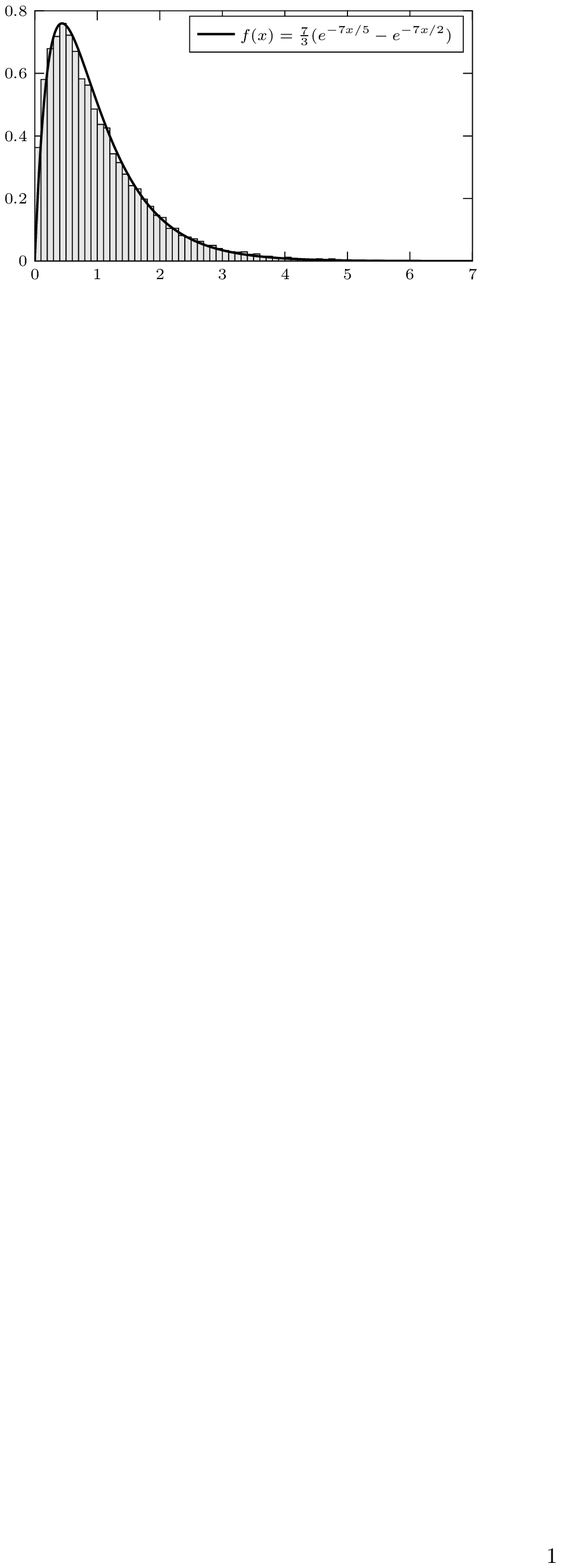}}
\hspace{0.5cm}
\subfigure[Case 2c*: $\underline{L}=(5,2,5)$, $\underline{a}=(1/2,3/2,1)$]{\includegraphics{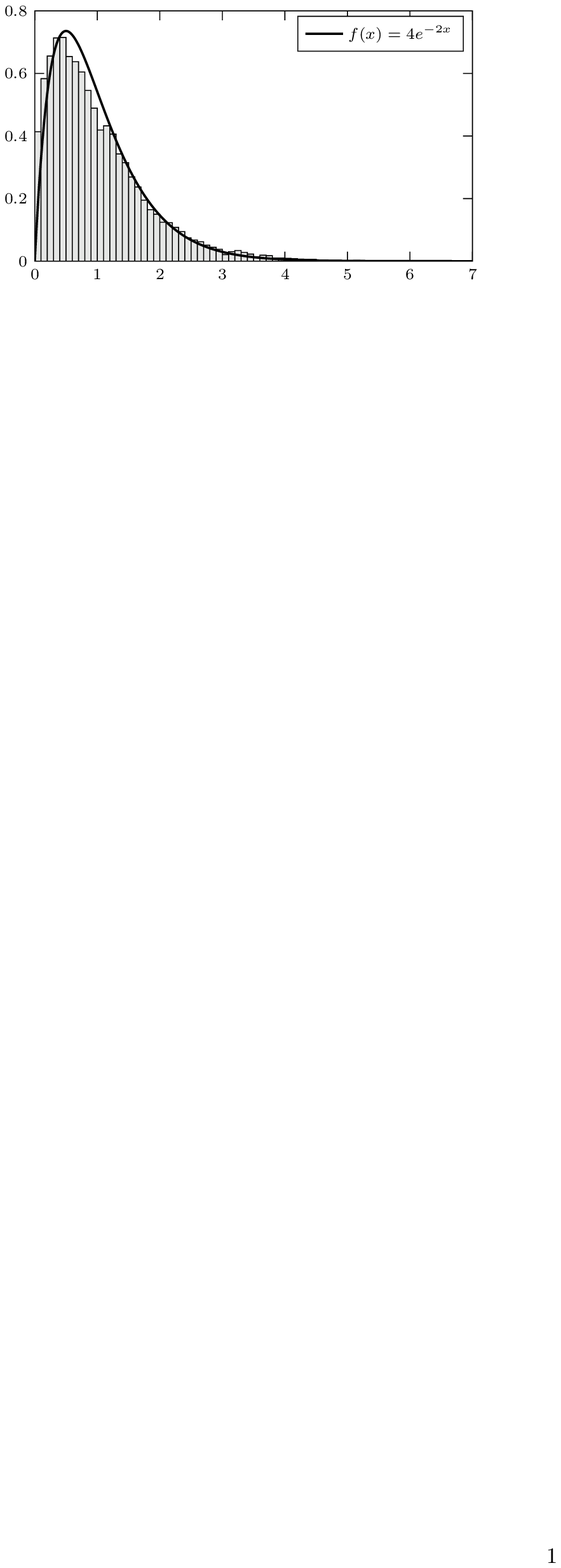}}
\\
\subfigure[Case 2d: $\underline{L}=(4,2,6)$, $\underline{a}=(3/5,6/5,3/5)$]{\includegraphics{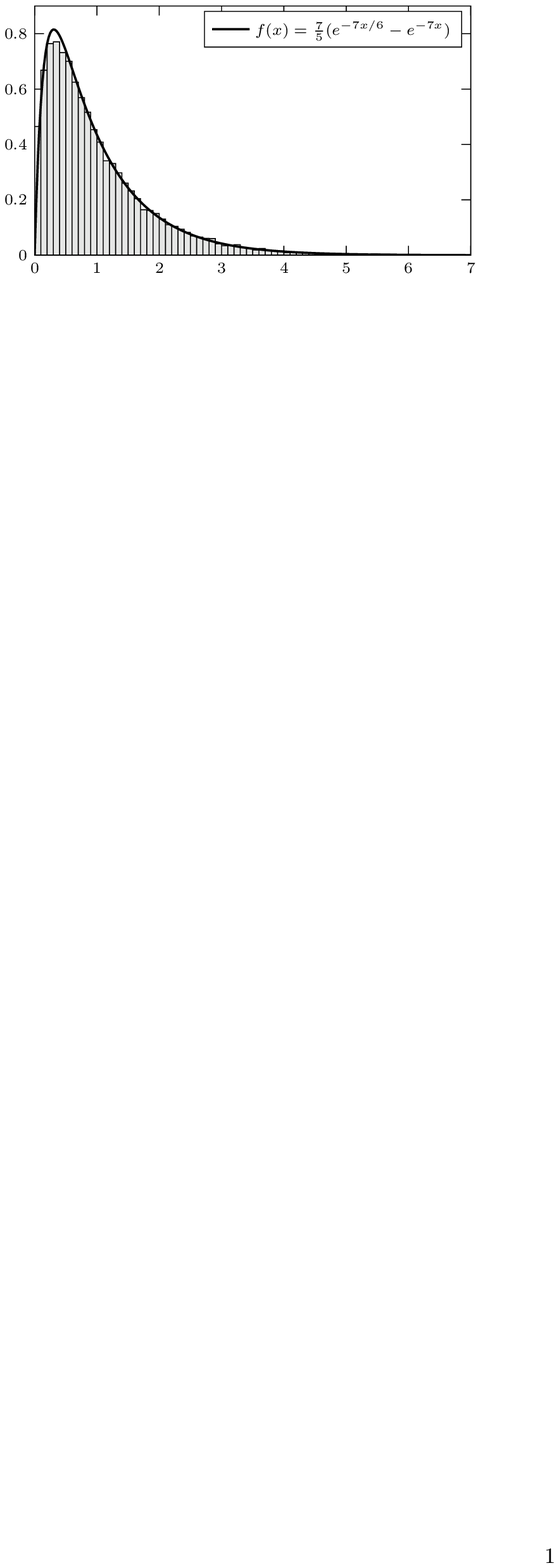}}
\hspace{0.5cm}
\subfigure[Case 3: $\underline{L}=(3,3,3)$, $\underline{a}=(1,3/4,3/2)$]{\includegraphics{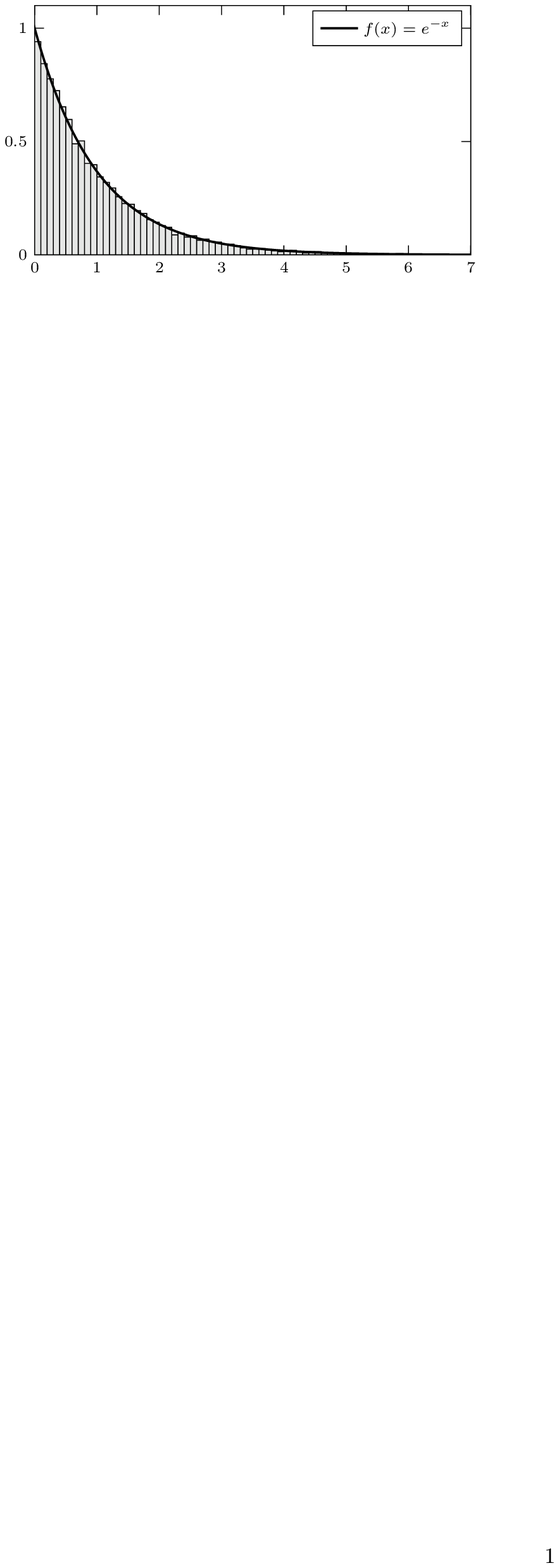}}
\caption{Plots of the empirical probability density function of the scaled transition times and the density $f(x)$ of $\alpha Y + (1-\alpha)W$.}
\label{fig:simnew2}
\end{figure}

\section{Model extensions}
\label{sec6}
So far we have assumed that two users interfere if and only if they belong to different components. In this section, we continue to assume that users that belong to different components interfere, but we allow users within the same component to interfere as well. If two or more users within component $C_k$ interfere with each other, there will be fewer admissible activity configurations of smaller size. In particular, not all the $L_k$ users of component $C_k$ can no longer be active simultaneously, which would then ease the transitions among different components.

In the previous sections, we further assumed all users within the same component to have the same activation rate, so that state aggregation could be applied to obtain an equivalent Markov process with a star-shaped state space. In this section, we allow the users within the same component to have possibly different activation rates. With minor abuse of notation, denote by $f_l(\nu)$ the activation rate of user~$l$, and define $F_k(\nu) := \sum_{l \in C_k} f_l(\nu)$ as the aggregate activation rate of all users in the $k$-th component. 

The components are assumed to be \textit{minimal}, in the sense that they cannot be split into two non-trivial components, while retaining the full interference across components. As before, each independent set of the conflict graph must be a subset of one of the components, because two users that belong to different components by definition interfere. However, some subsets within the same component may no longer be independent sets in the conflict graph.

For every $x \in \OS$ define $V_x \subseteq V$ to be the subset of users which are active in configuration $x$, i.e. $V_x:=\{i \in V : x_i=1\}$. For every $x \in \OS$, $V_x$ is by construction an independent set in the conflict graph $G$. Define moreover $\OO_k := \{x \in \OS : V_x \subseteq C_k, x \neq 0\}$. Then
\[
\OS=\{0\} \cup \bigcup_{k \in K} \OO_k.
\]
By construction, it follows that each component satisfies a certain monotonicity property: If $V_x \in \OO_k$, then for every nonempty $V_y \subseteq V_x$, we have $V_y \in \OO_k$. Indeed, if $x$ a feasible configuration and it belongs to $\OO_k$, then any configuration obtained from it by switching off some users is still feasible and belongs to $\OO_k$ as well. The next lemma shows that between any pair of activity states corresponding to the same component, there continues to exist a path which does not visit the state $0 \in \OS$.

\begin{lem}\label{lem:stillgood}
If $V_x, V_y \subseteq C_k$, $V_x, V_y \neq \emptyset$, then there exists a sequence of transitions between~$x$ and~$y$ that does not pass through the state $0 \in \OS$.
\end{lem}
\begin{proof}
If $V_x \cap V_y \neq \emptyset$, then the statement is trivially true. Suppose instead that $V_x \cap V_y = \emptyset$ and without loss of generality take $V_x=\{a_1,\dots,a_l\}$ and $V_y=\{b_1,\dots,b_m\}$. Define moreover $\mathcal C:=C_k \setminus (V_x \cup V_y) = \{ c_1,\dots,c_n\}$. Thanks to the monotonicity property of $C_k$, we have that $V_{a_1}, \dots, V_{a_l}, V_{b_1} \dots, V_{b_m}$ all belong to $C_k$ and each of them can be reached from $V_x$ and $V_y$, respectively, without passing through the state $0 \in \OS$. Suppose that every path from~$x$ to~$y$ passes through state $0 \in \OS$. Then none of the configurations $V_{\{a_k,b_j\}}$, $k=1,\dots,l$, $j=1,\dots,m$, belongs to $C_k$, i.e.~every user in $\{a_1,\dots,a_l\}$ interferes with every user in $\{b_1,\dots,b_m\}$.

We claim that every $c\in \mathcal C$ interferes either with all the users in $\{a_1,\dots,a_l\}$ or with all the users in $\{b_1,\dots,b_m\}$. Indeed, if there exist $a \in V_x$, $b \in V_y$ such that both $a$ and $b$ do not interfere with $c$, then we could construct a path from~$x$ to~$y$ which does not pass through state $0$, namely $V_x \to \dots \to \{a\} \to \{a,c\} \to \{c\} \to \{b,c\} \to \{b\} \to \dots \to V_y$. If there exists a user $c \in \mathcal C$ that interferes with all the users in $\{a_1,\dots,a_l\}$ \textit{and} $\{b_1,\dots,b_m\}$, then the component $C_k$ would not be minimal, since it can be split in $C_k \setminus \{c\}$ and $\{c\}$. Thus every $c\in \mathcal C$ interferes either with all the users in $\{a_1,\dots,a_l\}$ or with all the users in $\{b_1,\dots,b_m\}$, but not both. We can then consider two sets $A = V_x \cup \mathcal C_A$ and $B = V_y \cup \mathcal C_B$, with $\mathcal C_A \cap \mathcal C_B = \emptyset$ and $\mathcal C_A \cup \mathcal C_B = \mathcal C$, such that users in $A$ interfere with all users in $B$ and vice versa. Therefore $C_k = A \cup B$ is not a minimal component.
\end{proof}

Clearly state $0 \in \OS$ continues to be a bottleneck state which must be visited along any path between activity states corresponding to different components. For a user $l \in V$, denote by $e_l$ the configuration in $\Omega$ where only the user $l$ is active. Clearly, $l \in C_k$ if and only if $e_l \in \Omega_k$. For any two states $x, y \in \OS$, let $T_{x,y} = \inf\{t \geq 0: \xt = y | X(0) = x\}$ be a random variable representing the transition time from state~$x$ to state~$y$, i.e.~the hitting time of state~$y$ starting in state~$x$.

Let $x, y \in \OS$ be two activity states, with $V_x \subseteq C_{k_1}$ and $V_y \subseteq C_{k_2}$, $k_1 \neq k_2$. In order to give a stochastic representation of the transition time $T_{x,y}$, similar in spirit to \eqref{eq:kpartiterepr}, we define the following random variables for $l \in V \setminus ( \{0\} \cup C_{k_2})$:
\begin{itemize}
\item $N_l(\nu)$: number of times the process makes a transition $0 \to e_l \in \Omega_k$, $k \neq k_2$, before the first transition to $C_{k_2}$ occurs;
\item $\hat{T}_{0, e_l}^{(i)}(\nu)$: time spent in state~0 before the $i$-th transition to state $e_l \in \Omega_k$, with $k \neq k_2$, $i = 1, \dots, N_l(\nu)$;
\item $T_{e_l, 0}^{(i)}(\nu)$: time to return to state~$0$ after the $i$-th transition to state $e_l \in \Omega_k$, with $k \neq k_2$, $i = 1, \dots, N_l(\nu)$.
\end{itemize}
Moreover, for $l \in C_{k_2}$, define $\hat{T}_{0, e_l}(\nu)$ as the time spent in state~$0$ before the first transition to state $e_l \in \Omega_{k_2}$. Lemma~\ref{lem:stillgood} implies that the transition time $T_{x,y}$ may be represented as
\eqn{\label{eq:represen1}
T_{x,y} = T_{x,0} + \sum_{k \neq k_2} \sum_{l \in C_k} \sum_{i = 1}^{N_l} (\hat{T}_{0,e_l}^{(i)} + T_{e_l,0}^{(i)}) + \sum_{l \in C_{k_2}} I_l (\hat{T}_{0,e_l} + T_{e_l,y}),
}
where $I_l$, $l \in C_{k_2}$, are $0$-$1$ variables with $\sum_{l \in C_{k_2}} I_l = 1$ and $\pr{I_l = 1} = f_l(\nu) / F_{k_2}(\nu)$, $l \in C_{k_2}$, and the variables $T_{e_l,0}$, $l \in C_k$, are transition times when considering the component $C_k$ in isolation. Moreover, in the above representation the dependence on the parameter~$\nu$ is suppressed for compactness and all the random variables representing time durations are mutually independent as well as independent of the random variables $N_l(\nu)$, $l \in V \setminus ( \{0\} \cup C_{k_2})$.

In order to determine the asymptotic behavior of the transition time $T_{x,y}(\nu)$ as $\ninf$, we now proceed to analyze the asymptotic behavior of the escape times $T_{e_l,0}$, $l \in C_k$, in the stochastic representation. Unless stated otherwise, we henceforth let $z \in \OO_k$ and focus on the Markov process $\xstb$ restricted to the state space $\OO_k^+ = \OO_k \cup \{0\}$.

The steady-state probability of a state $u \in \OO_k^+$ is
\[
\pi_u(\nu) = \frac{1}{Z_k(\nu)} \prod\limits_{l \in C_k} f_l(\nu)^{u_l},
\]
with normalization constant
\[
Z_k(\nu) =
\sum_{u' \in \OO_k^+} \prod\limits_{l \in C_k} f_l(\nu)^{u_l'}.
\]
Define
\[
g_k(\nu) := \max_{u \in \OO_k} \prod\limits_{l \in C_k} f_l(\nu)^{u_l} \quad \text{ and } \quad \eta_k := \min_{l \in C_k} \limnu \log_\nu f_l(\nu).
\]
We make the mild technical assumptions that $\eta \in (0,\infty)$ and that $\psi_k = \limnu Z_k(\nu) / g_k(\nu)$ exists. Then the following two asymptotic properties of the escape time $T_{z,0}(\nu)$ can be established:
\eqn{\label{eq:expected1}
\E T_{z,0}(\nu) \sim \frac{\psi_k g_k(\nu)}{F_k(\nu)},\quad \text{ as } \ninf,
}
and
\eqn{\label{eq:scaled1}
\frac{T_{z,0}(\nu)}{\E T_{z,0}(\nu)} \cd \rmexp(1),\quad \text{ as } \ninf.
}

In order to provide a brief sketch of the proof arguments, we first introduce some further useful notation. Let $N_{z,0}(\nu)$ be a random variable representing the number of visits to state~0 in between two consecutive visits to state~$z$. Let $R_z(\nu)$ be the residence time in state~$z$ and $T_{z,z}^+(\nu)$ the first return time to state~$z$. Noting that
\[
\pi_z(\nu) = \frac{\E R_z(\nu)}{\E T_{z,z}^+(\nu)}, \quad \E R_z(\nu) \leq 1 \quad \text{ and } \quad \pi_z(\nu) = \frac{1}{Z_k(\nu)} \prod\limits_{l \in C_k} f_l(\nu)^{z_l},
\]
we obtain
\[
\frac{\nu^\eta \E T_{z,z}^+(\nu)}{\psi_k g_k(\nu)} = \frac{\nu^\eta \E R_z(\nu)}{\pi_z(\nu) \psi_k g_k(\nu)} \leq \frac{ \nu^\eta}{ \prod\limits_{l \in C_k} f_l(\nu)^{z_l} } \frac{ Z_k(
\nu)}{ \psi_k g_k(\nu)}= o(1), \quad \text{ as } \ninf.
\]
Using similar arguments as in~\cite{OV05}, it may be shown that
\[
\pr{T_{u, z} (\nu) > g_k(\nu) \nu^{- \eta_k / 2}} \leq r < 1
\]
for all states~$u$ with $V_u \in \B_k$, implying (by the strong Markov property)
\[
\pr{T_{u, z} (\nu) > n g_k(\nu) \nu^{- \eta_k / 2}} \leq r^n,
\]
and that $\pr{T_{0, z}(\nu) < T_{0,0}^+(\nu)} \geq s(\nu)$, with $s(\nu) \to 1$ as $\ninf$. This means that after a visit to state~0, the number of additional visits to that state before the first visit to state~$z$ is stochastically bounded from above by a geometrically distributed random variable with parameter $1 - s(\nu)$. This implies
\eqn{\label{eq:single1}
\E N_{z,0}(\nu) \sim \pr{N_{z,0}(\nu) \geq 1} \quad  \text{ as } \ninf.
}
It may then be deduced that the distribution of $T_{z,z}^+(\nu)$ satisfies the uniform integrability condition in~\cite{GR05}. Theorem~1 in~\cite{GR05} then yields the asymptotic exponentiality property in~\eqref{eq:scaled1} and
\[
\E T_{z,0}(\nu) \sim \frac{\E T_{z,z}^+(\nu)}{\pr{N_{z,0}(\nu) \geq 1}}.
\]
Observing that
\[
\frac{\pi_0(\nu) }{ \E R_0(\nu) }= \E N_{z,0}(\nu) \frac{ \pi_z(\nu) }{ \E R_z(\nu)},
\]
and invoking~\eqref{eq:single1}, we deduce that the term in the right-hand side asymptotically behaves as
\[
\frac{\E T_{z,z}^+(\nu)  \pi_z(\nu) \E R_0(\nu) }{\pi_0(\nu) \E R_z(\nu)}  = \frac{\E R_0(\nu)}{\pi_0(\nu)} = \frac{Z_k(\nu)}{F_k(\nu)} \sim \frac{\psi_k g_k(\nu)}{F_k(\nu)},
\]
yielding~\eqref{eq:expected1} as stated. \vspace{0.5cm}

The two asymptotic properties~\eqref{eq:expected1},~\eqref{eq:scaled1} for the order-of-magnitude and the scaled distribution of the escape time $T_{z,0}(\nu)$ mirror those stated in~\eqref{eq:sdp} and Proposition~\ref{prop:ael0}. Using these two properties and the stochastic representation~\eqref{eq:represen1}, similar results can be established for the asymptotic behavior of the transition time $T_{x,y}(\nu)$ as in Theorems~\ref{thm:thm1} and~\ref{thm:thm2}. For any $l \in C_k$, define
\[
\Theta_l(\nu) = \frac{f_l(\nu) g_k(\nu)}{F_k(\nu)}.
\]
In this case the set $K^*$ needs to be defined as those $l \in \bigcup_{k \neq k_2} C_k$ such that $\lim_{\nu \to \infty} \Theta_l(\nu) / \Theta_m(\nu) > 0$ for all $m \in \bigcup_{k \neq k_2} C_k$. Also, additional conditions need to be imposed in order to ensure that
\[
\sum_{l \in C_{k_2}} \frac{f_l(\nu)}{F_{k_2}(\nu)} \E T_{e_l,y}(\nu) = o(\E T_{x, y}(\nu)), \quad  \text{ as } \ninf,
\]
which guarantees that the expected time to reach state $y$, once the process hits the target component $C_{k_2}$, is asymptotically small with respect to the overall transition time $T_{x, y}(\nu)$.

\section{Throughput starvation and near-saturation}
\label{sec7}
In this section we return to a complete partite networks with equal activation rates for users in the same component. We show how the results for the asymptotics of the transition time $\TT(\nu)$ in Theorems~\ref{thm:thm1} and~\ref{thm:thm2} can be exploited to gain insight about phenomena like throughput starvation or near-saturation. More specifically, in Subsection~\ref{sec71} we present the proof of Theorem~\ref{thm:thm3}, which gives an asymptotic lower bound on the probability of throughput starvation, while in Subsection~\ref{sec72} we prove Proposition~\ref{prop:tails}, a complementary result which indicates over what time scales throughput near-saturation occurs.

\subsection{Proof of Theorem~\ref{thm:thm3}}
\label{sec71}
Observe that $\tau_{k_2}(t(\nu)) > 0$ implies $t(\nu) > \TTo(\nu)$, because the throughput of branch $\B_{k_2}$ remains zero until the activity process enters $\B_{k_2}$. Hence
\[
\pr{\tau_{k_2}(t(\nu)) > 0} \leq \pr{\TTo(\nu) < t(\nu)} = \pr{\frac{\TTo(\nu)}{\E \TTo(\nu)} < \frac{t(\nu)}{\E \TTo(\nu)}}.
\]
Taking the limit as $\ninf$, Theorem~\ref{thm:thm2}, gives $\limnu \pr{\tau_{k_2}(t(\nu)) > 0} \leq \pr{ Z < \o }$, and~\eqref{eq:starvation} follows.

\subsection{Throughput near-saturation}
\label{sec72}
Assume that at time $t=0$ there is at least one user active in $C_k$, i.e.~$X(0)=(k,l) \in \B_k$. Define the total full-component active time in $[0,t]$ as
\[
\t_k[0,t]:=\int_0^t I_{\{ X(s)=(k,L_k) \}} \, ds,
\]
the residual time in $C_k$ during $[0,t]$ as
\[
R_{k}[0,t]:=\int_0^t I_{\{ X(r) \in \B_k \, \forall \, r \in [0,s] \}} \, ds,
\]
and the full-component active time contained in the residual time in $C_{k}$ during $[0,t]$ as
\[
\tr_k[0,t]:=\int_0^t I_{\{ X(r) \in \B_k \, \forall \, r \in [0,s] \}}  I_{\{ X(s)=(k,L_k) \}} \, ds.
\]

For compactness, we have suppressed the implicit dependence on the parameter $\nu$ and the initial state $(k,l)$ in the notation. From this point onwards, we will also drop the subscript $k$ to keep the notation light.

Note that $R[0,t]\ed \min \{t,T_{(k,l),0}\}$ and that $\tr[0,t] \ed \t[0, R[0,t]]$. The random variables $\t[0,t]$, $R[0,t]$ and $\tr[0,t]$, being particular occupancy times, are non-decreasing in $t$ on every sample path of the activity process $\xtb$. Therefore, the random variables
\[
\t[0,\infty]:=\lim_{t \to \infty} \t[0,t], \quad R[0,\infty]:=\lim_{t \to \infty} R[0,t]=T_{(k,l),0}, \quad \tr[0,\infty]:=\lim_{t \to \infty} \tr[0,t]
\]
are well defined. For $0\leq s \leq t \leq \infty$, we define
\[
\t[s,t]:=\t[0,t]-\t[0,s], \quad R[s,t]:=R[0,t]-R[0,s], \quad \tr[s,t]:=\tr[0,t]-\tr[0,s].
\]
From the above definition, it is easily seen that for every sample path, $\tr[s,t]$ provides a lower bound for both $\t[s,t]$ and $R[s,t]$, as stated in the next lemma.
\begin{lem}\label{lem:dom}
For $0\leq s\leq t \leq \infty$, $\tr[s,t] \leq \t[s,t]$ and $\tr[s,t]\leq R[s,t]$.
\end{lem}
\begin{proof}
Rearranging terms, the differences $\t[s,t] - \tr[s,t]$ and $R[s,t]-\tr[s,t]$ can both be written as integrals with an integrand that is always non-negative.
\end{proof}

In particular, Lemma~\ref{lem:dom} implies that, for every $0\leq s\leq t \leq \infty$
\eqn{\label{eq:fmi}
\E \tr[s,t] \leq \E R[s,t].
}
However, as stated in the next lemma, in the limit as $\ninf$, the ratio of the expected values of $\tr[0,\infty]$ and $T_{(k,l),0}=R[0,\infty]$ converges to $1$.
\begin{lem} \label{lem:inftye}
For any initial state $X(0)=(k,l) \in \B_k$,
\[
\limnu \frac{\E \tr[0,\infty]}{\E T_{(k,l),0}} =1.
\]
\end{lem}
\begin{proof}
Since the ratio is clearly less than $1$ by Equation~\eqref{eq:fmi}, it suffices to show that the liminf as $\ninf$ is larger than $1$. Applying the result in~\cite{SW00} and using~\eqref{eq:sd}, one obtains that for every $1 \leq l \leq L_k$, if $X(0)=(k,l)$, then
\[
\E \tr[0,\infty] = \E \Big( \int_{0}^{T_{(k,l),0}} I_{\{ X(s)=(k,L_k) \}} \, ds \Big) \geq \frac{1}{L_k} f_k(\nu)^{L_k-1},
\]
and thus, involving~\eqref{eq:sdp},
\[
\liminfnu \frac{\E \tr[0,\infty]}{\E T_{(k,l),0}} \geq \liminfnu \frac{f_k(\nu)^{L_k-1}/L_k}{\E T_{(k,l),0}}  = 1.
\]
\end{proof}

The next proposition establishes a near-saturation property in the sense that if $X(0)=(k,l) \in B_k$, then for any time period $t(\nu)=o(\E T_{(k,l),0})$ every user in $C_k$ will be active an arbitrarily large fraction of the time with probability one as $\ninf$.

\begin{prop}\label{prop:tails}
Suppose that $X(0)=(k,l) \in \B_k$ and that $T_{(k,l),0}/\E T_{(k,l),0} \cd Z$ as $\ninf$. Then for every $\o \in [0,1]$ and every $\d>0$,
\[
\liminfnu \pr{\t[0, \o \E T_{(k,l),0}] \geq (1-\d) \o \E T_{(k,l),0}} \geq \pr{Z \geq \o}.
\]
In particular, for any $t(\nu)=o(\E T_{(k,l),0}(\nu))$, $\liminfnu \pr{\t[0,t(\nu)] \geq (1-\d) t(\nu)} =1$.
\end{prop}

\begin{rem} As mentioned earlier, the hypothesis that $T_{(k,l),0}/ \E T_{(k,l),0} \cd Z$ is not just a convenient assumption, but something that we actually know. In particular, Proposition~\ref{prop:ael0} says that $Z \ed \rmexp(1)$. Moreover, since the result holds for every initial state in $\B_k$, it is true also for a random initial state in $\B_k$. Indeed, as seen in Section~\ref{sec4}, the convergence in distribution of $T_{(k,l),0}(\nu)/ \E T_{(k,l),0}(\nu)$ to $Z$ as $\ninf$ does not depend on the initial state, as long as it belongs to $\B_k$.
\end{rem}

\begin{proof}
In order to keep the notation light, we denote in this proof the hitting time $T_{(k,l),0}$ by $T$. Firstly, Lemma~\ref{lem:dom} implies that
\[
\pr{\t[0, \o \E T] \geq (1-\d) \o \E T } \geq \pr{\tr[0,\o \E T] \geq (1-\d) \o \E T_0 }.
\]
Moreover, by definition of $R[0,t]=\min\{t,T\}$, we have
\[
\pr{R[0, \o \E T] \geq \o \E T} = \pr{T \geq \o \E T}.
\]
In view of the hypothesis that $T(\nu)/ \E T(\nu) \cd Z$ as $\ninf$, it therefore suffices to prove that for every $ \o \in [0,1]$ and every $\d >0$,
\[
\liminfnu \pr{\tr[0,\o \E T] \geq (1-\d)\o \E T} \geq \liminfnu \pr{R[0,\o \E T] \geq \o \E T}.
\]
Suppose that this latter statement is false, i.e.~there exist $\o_0 \in [0,1]$, $\d >0$ and $\eta>0$ such that
\eqn{\label{eq:eta}
\liminfnu \pr{\tr[0,\o_0 \E T] \geq (1-\d)\o_0 \E T} \leq \liminfnu \pr{R[0,\o_0 \E T] \geq \o_0 \E T} -\eta.
}
Then it can be shown that there exists $\e_{\o_0,\d} >0$ such that
\eqn{\label{eq:1a}
\liminfnu \frac{\E \tr[0, \o_0\E T] }{\E T} \leq \liminfnu \frac{\E R[0, \o_0\E T]}{\E T}-\e_{\o_0,\d}.
}
Indeed,
\eqan{
\liminfnu  & \left( \frac{\E R[0, \o_0\E T]}{\E T} - \frac{\E \tr[0, \o_0\E T]}{\E T} \right ) = \liminfnu \int_0^\infty \pr{\frac{R[0,\o_0 \E T]}{\E T} \geq y} - \pr{ \frac{\tr[0,\o_0 \E T]}{\E T} \geq y} \, dy \nonumber \\
& \stackrel{\text{Lemma}~\ref{lem:dom}}{\geq} \liminfnu \int_{(1-\d)\o_0}^{\o_0}\pr{\frac{R[0,\o_0 \E T]}{\E T} \geq y}- \pr{\frac{\tr[0,\o_0 \E T]}{\E T} \geq y} \, dy \nonumber \\
& \geq \int_{(1-\d)\o_0}^{\o_0} \liminfnu \left ( \pr{\frac{R[0,\o_0 \E T]}{\E T} \geq y}- \pr{\frac{\tr[0,\o_0 \E T]}{\E T} \geq y} \right )\, dy \nonumber \\
& \geq \eta \delta \o_0 >0, \nonumber
}
where the second last inequality follows from the generalized Fatou's lemma, while the last inequality follows from~\eqref{eq:eta} and is illustrated by Figure~\ref{fig:eta}. Thus we can take $\e_{\o_0,\d}:=\eta \delta \o_0$.
\begin{figure}[!ht]
\centering
\includegraphics{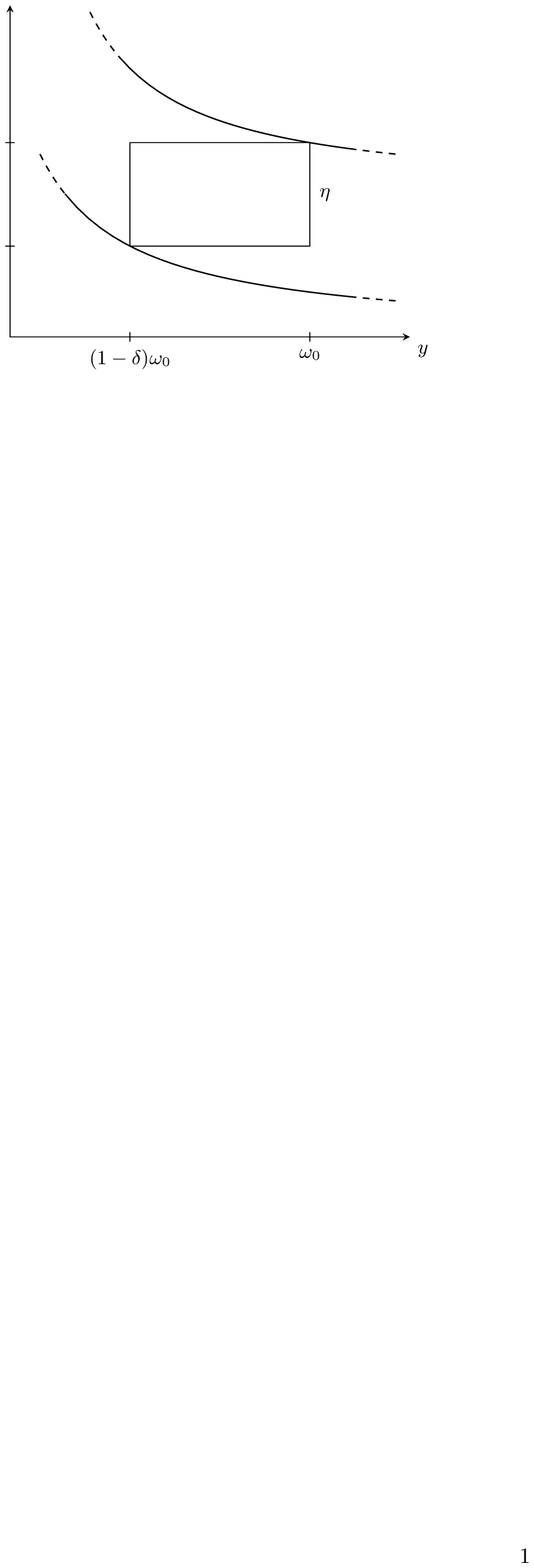}
\caption{$\pr{ \tr[0,\omega_0  \E T] / \E T \geq y}$ (lower line) vs $\pr{R[0,\omega_0  \E T] / \E T \geq y}$ (upper line)}\label{fig:eta}
\end{figure}

Equation~\eqref{eq:fmi} yields
\eqn{\label{eq:1b}
\liminfnu \frac{\E \tr[\o_0\E T, \infty] }{\E T} \leq \liminfnu \frac{\E R[\o_0\E T,\infty]}{\E T},
}
and thus, summing term by term~\eqref{eq:1a} and~\eqref{eq:1b}, since by definition $\E R[0,\infty]=\E T$,
\[
\liminfnu \frac{\E \tr[0, \infty] }{\E T} \leq \liminfnu \frac{\E R[0,\infty]}{\E T}-\e_{\o_0,\d}=1-\e_{\o_0,\d},
\]
which contradicts Lemma~\ref{lem:inftye}.
\end{proof}

\section{Mixing times}
\label{sec8}

In the previous sections we have analyzed the transient behavior of the Markov process $\xtb$ in terms of hitting times and we have shown how this leads to starvation of individual users over certain time scales. In this section we turn attention to the long-run behavior of the Markov process $\xtb$ and in particular examine the rate of convergence to the stationary distribution. We measure the rate of convergence in terms of the total variation distance and the so-called \textit{mixing time}, which describes the time required for the distance to stationarity to become small.

The mixing time becomes particularly relevant when the network has two or more dominant components which together attract the entire probability mass in the limit as $\ninf$. Indeed, in this case, the mixing time provides an indication of how long it takes the activity process to reach a certain level of fairness among the dominant components. We will prove a lower bound for the mixing time using the notion of conductance.

The maximal distance over $x\in\OO$, measured in terms of total variation, between the distribution at time~$t$ and the stationary distribution is defined as
\[
d(t,\nu):= \max_{x\in\OO}\|\pr{X(t) \in \cdot \,| X(0)=x }-\pi(\nu)\|_{\mathrm{TV}}.
\]
We define the mixing time of the process $\xtb$ as
\[
\tm(\e,\nu):=\inf\{t \geq 0 : d(t,\nu)\leq \e\}.
\]
For a fixed $r \in (0,1)$ consider the subset $\tilde K(r)$ of branches whose stationary probability is asymptotically no more than $r$, i.e. $\tilde K(r):=\{ k ~:~ \limnu \pi_{\B_\kappa}(\nu) \leq r \}$. Define $\kappa \in \tilde K(r)$ as the index corresponding to the branch $\B_\kappa$ which has asymptotically the largest mean escape time, i.e.~such that for every $j \in \tilde K(r)$,
\eqn{\label{eq:kappa}
\limnu \frac{\E T_{(\kappa,1),0}(\nu)}{\E T_{(j,1),0}(\nu)} = \limnu \frac{L_j f_{\kappa}(\nu)^{L_{\kappa}-1}}{L_{\kappa} f_j(\nu)^{L_j-1}} \geq 1.
}

The next result shows that the mixing time is asymptotically at least of the same order of magnitude as the escape time from branch $\B_\kappa$.

\begin{prop}
\label{prop:mix}
For any $r\in (0,1)$ and $\e \in \left (0, \frac{1-r}{2} \right )$, there exists a constant $C_{\e,r} > 0$ and $\nu_0> 0$ such that for every $\nu > \nu_0$
\[
\tm(\e,\nu) \geq C_{\e,r} \frac{f_{\kappa}(\nu)^{L_{\kappa}-1}}{L_\kappa}.
\]
\end{prop}

Proposition~\ref{prop:mix} shows that it can take an extremely long time for the process $\xtb$ to reach stationarity, especially when $\nu$ is large. Such a long mixing time is typically due to the activity process being stuck for a considerable period in one of the components, and thus not visiting the states in the other components. We will prove Proposition~\ref{prop:mix} exploiting the presence of a bottleneck in the state space and using the notion of conductance.

For $S \subseteq \OO$, let $\pi_S(\nu)=\sum_{(k,l) \in S } \pi_{(k,l)}(\nu)$ be the stationary probability of $S$. Define the \textit{probability flow out of} $S$ as
\[Q(S,S^c):=\sum_{(k,l) \in S, (j,m) \in S^c} \pi_{(k,l)}(\nu) q((k,l),(j,m))\]
and its \textit{conductance} as $\Phi(S):=Q(S,S^c)/\pi_S$. All the quantities we just defined clearly depend on $\nu$, but we suppress it for compactness. The \textit{conductance profile} of the process $\xtb$ is defined as
\[\Phi_r(\nu):=\min_{S \,: \, \pi(S) \leq r} \Phi(S).\]

The following result, valid for any Markov process on a finite state space $\OO$ with conductance profile $\Phi_r$, shows how the conductance of the process yields a lower bound on the mixing time. It is a continuous-time version of Theorem 7.3 in~\cite{LPW09} and the proof is relegated to~\ref{ap4}.
\begin{lem} \label{lem:cond}
For any $r\in (0,1)$ and any $\e \in \left(0,\frac{1-r}{2}\right)$,
\[
\tm(\e, \nu) \geq \frac{1-r - 2 \e}{\Phi_r (\nu)}.
\]
\end{lem}

In order to get a sharp bound for the conductance and hence a sharp lower bound for the mixing time, we need to identify a subset $S$ with low conductance. As proved in~\cite{LS88}, it suffices to look at the connected subsets of the state space. Therefore the branches in $\tilde{K}(r)$ become naturally good candidates for being the lowest-conductance subsets of $\OO$. From~\eqref{eq:sd} it follows that if $(k,l)$ and $(k,m)$ are states in the same branch $\B_k$, then
\[
\frac{\pi_{(k,m)}(\nu)}{\pi_{(k,l)}(\nu)} =  \frac{l!(L_k-l)!}{m!(L_k-m)!} f(\nu)^{m-l}, \quad \text{ as } \ninf.
\]
Thus the conductance of $\B_k$ satisfies
\[
\Phi(\B_k) = \frac{\pi_{(k,1)}(\nu) \cdot 1 }{ \sum_{l=1}^{L_k} \pi_{(k,l)}(\nu)}= \frac{ \frac{\pi_{(k,1)}(\nu)}{\pi_{(k,L_k)}(\nu)}}{ \sum_{l=1}^{L_k} \frac{\pi_{(k,l)}(\nu)}{\pi_{(k,L_k)}(\nu)}} \sim L_k f_k(\nu)^{1-L_k}, \quad  \text{ as } \ninf.
\]
Thanks to the definition of $\kappa$, $\B_\kappa$ is asymptotically the branch with the smallest conductance, see Equation~\eqref{eq:kappa}. Since by definition $ \Phi_r(\nu) \leq \Phi(\B_{\kappa})$, Lemma~\ref{lem:cond} then gives that for every $\e \in (0,\frac{1}{4})$ and for $\nu$ sufficiently large
\[
\tm(\e,\nu) \geq ( 1-r - 2 \e) \frac{f_{\kappa}(\nu)^{L_{\kappa}-1}}{L_\kappa},
\]
which completes the proof of the lower bound claimed in Proposition~\ref{prop:mix}.

\section{Conclusions}
\label{sec9}
We have studied hitting times and mixing properties in dense wireless random-access networks. We have represented the activity processes in such networks in terms of Markov processes on complete partite graphs. In particular, in dense networks, high activity rates lead to network behavior in which users in maximum independent sets coalesce into components which compete for the wireless medium. We have shown that components monopolize the wireless medium for extremely long periods, which leads to long mixing times and starvation of all other components. Hence, users in a particular component alternate between enjoying long periods of access, and facing long periods of starvation. 

While the slow nature of the transitions is a common characteristic, the asymptotic distribution of the scaled transition time depends crucially on the structure of the network and on the initial and target components and there is a notable variety of possible scenarios. In particular, in some scenarios, the distribution of the scaled transition time is non-exponential. This is due to the heterogeneous activation rates among components in conjunction with the presence of intermediate components where the activity process resides for long periods along the transition. 

The complete partite graphs that we focused on in the present paper, are arguably the worst possible networks in terms of transition times and starvation effects, given the size of the network. Indeed, the fact that the users are grouped into components, with no interference within components and full interference between components, turns out to be a key element for starvation to occur. This is reflected in the fact that the transition times exhibit exponential growth in the component size. Graphs that are non-partite, or partite but not complete, will have a less extreme tendency for starvation, although the issue may still arise to a milder degree. For example, interference between nodes inside the same component will reduce the size of the maximum independent set, and lower the likelihood of the maximal activity states relative to the bottleneck state where all the nodes are inactive. This is borne out by the model extensions in Section~\ref{sec6} where the order-of-magnitude of the transition time is governed by the maximum independent set within a component. Likewise, lack of interference between nodes in different components will result in bottleneck states where some of the nodes may be active, and raise the likelihood of the bottleneck state relative to the dominant activity states. This is illustrated by the work in~\cite{ZBvLN13} which investigates the transition times and delays in a toric grid. The toric grid is a bi-partite graph, but with fewer edges between the two components. The results in~\cite{ZBvLN13} show that the delays and long transition times, while still severe, are of a lower order than for the complete bi-partite graph.

\bibliographystyle{plain}
\bibliography{cpnfinal.bbl}

\appendix
\section{Proof of Lemma \ref{lem:asymptotic1}}
\label{ap1}
We will prove a slightly more general version of Lemma~\ref{lem:asymptotic1}, assuming the rates of the process are those described in Subsection~\ref{subsec:gencoef}. Order the state space as $\OO=\{L,L-1,\dots, 1, 0\}$ and consider the generator matrix $Q(\nu)$ of the process $\xtb$ with $0$ an absorbing state. That is,
\[
Q_L(\nu)=
\begin{bmatrix}
 q_{L} & d_{L}& 0 & &\dots & &0 \\
 a_{L-1}f(\nu)& q_{L-1} & d_{L-1} & 0 &   \\
 0& a_{L-2}f(\nu) & q_{L-2} & d_{L-2} & 0 \\
 \\
 \vdots & &\ddots & \ddots & \ddots & & \\
 \\
  & & &0 & a_{1}f(\nu)  &  q_{1} & d_{1}\\
  0 & \dots & & 0 & 0 & 0 & 0\\
\end{bmatrix},
\]
where the diagonal elements are $q_l(\nu)=- a_l f(\nu) -d_l$ for $l=1,\dots, L_1$ and $q_L(\nu)=-d_L$. The matrix $Q(\nu)$ can be written as
\[
Q(\nu)=\left(\begin{array}{cc}\mathbf{T}(\nu) & \mathbf{t}(\nu) \\ \mathbf{0} & 0 \\\end{array}\right),
\]
where $\mathbf{T}(\nu)$ is an $L \times L$ invertible matrix. Since the characteristic polynomials of $-Q(\nu)$ and $-\mathbf{T}(\nu)$ satisfy the relation $ p_{-Q(\nu)}(z)= - z \, p_{-\mathbf{T}(\nu)}(z)$, the spectrum of $-Q(\nu)$ consists of that of $-\mathbf{T}(\nu)$ plus the eigenvalue zero with multiplicity one.\\
Denote by $D(\nu)$ the $L\times L$ diagonal matrix, whose diagonal entries are $\{\sqrt{\xi_l(\nu)}\}_{i=L}^1$, where the $\xi$'s are the so-called \textit{potential coefficients}, defined recursively as
\[
\xi_L(\nu)=1 \quad \text{ and } \quad \xi_{l-1}(\nu)=\frac{d_l}{a_l f(\nu)}\xi_{l}(\nu), \quad l=L-1,\dots,1.
\]
The $L\times L$ matrix $G(\nu)=- D(\nu)^{1/2} \mathbf{T}(\nu) D(\nu)^{-1/2}$ is tridiagonal and symmetric with diagonal entries $G_{l,l}(\nu)=q_{L-l+1}(\nu)$ and $G_{l,l+1}(\nu)=g_{l+1,l}(\nu)=-\sqrt{d_l a_{l-1} f(\nu)}$. Since $G(\nu)$ is similar to $-\mathbf{T}(\nu)$, they have the same spectrum. Denote by $\mathcal{D}(p,R)$ the closed disc centered in $p$ with radius $R$, i.e.~$\mathcal{D}(p,R)=\{z \in \mathbb C ~:~ |z-p| \leq R\}$. Consider the \textit{Gershgorin discs} $ \{\mathcal{D}_l(\nu)\}_{l=1}^L$ of $G(\nu)$, defined as $\mathcal{D}_l(\nu):=\mathcal{D}(-q_l(\nu),R_l(\nu))$, where the radius $R_l(\nu)$ is the sum of the absolute values of the non-diagonal entries in the $L-l+1$-th row, i.e.~$R_l(\nu):=\sum_{m \neq L-l+1} |G_{L-l+1,m}(\nu)|$. Then
\eqan{
\mathcal{D}_{L}(\nu)&=\mathcal{D}(d_L,\sqrt{d_L a_{L-1} f(\nu)}),\nonumber\\
\mathcal{D}_{L-1}(\nu)&=\mathcal{D}(d_{L-1}+a_{L-1} f(\nu),\sqrt{d_L a_{L-1} f(\nu)}+\sqrt{d_{L-1}a_{L-2} f(\nu)}), \nonumber\\
& \, \, \, \vdots  \nonumber\\
\mathcal{D}_{2}(\nu)&=\mathcal{D}(d_{2}+a_{2} f(\nu),\sqrt{d_3 a_{2} f(\nu)}+\sqrt{d_{2}a_{1} f(\nu)}),\nonumber\\
\mathcal{D}_{1}(\nu)&=\mathcal{D}(d_1+a_1 f(\nu),\sqrt{d_2 a_1 f(\nu)}).\nonumber
}
We now exploit the second Gershgorin circle theorem, which is reproduced here for completeness.
\begin{thma}
If the union of $j$ Gershgorin discs of a real $r\times r$ matrix $A$ is disjoint from the union of the other $r - j$ Gershgorin discs, then the former union contains exactly $j$ and the latter the remaining $r - j$ eigenvalues of $A$.\end{thma}

In our case, for $\nu$ sufficiently large, the disc $\mathcal{D}_{L}(\nu)$ does not intersect with the union $\bigcup_{l=1}^{L-1} \mathcal{D}_l(\nu)$, thus the smallest eigenvalue $\theta_1(\nu)$ lies in $ \mathcal{D}_{L}(\nu)$ and the other $L-1$ ones in $\bigcup_{l=1}^{L-1} \mathcal{D}_l(\nu)$. Hence, for $\nu$ sufficiently large, the following inequalities hold
\eqan{
\theta_1(\nu) &\leq A+B \sqrt{f(\nu)}, \nonumber \\
\theta_i(\nu) &\geq C f(\nu)  - D \sqrt{f(\nu)}, \quad i=2,\dots,L, \nonumber
}
where $A,B,C,D \in \R_+$ and, more precisely, $A=d_L$, $B=\sqrt{d_L a_{L-1}}$, $C=\min_{l=1,\dots,L-1} a_l$ and $D=\max\{ \sqrt{d_L a_{L-1}} + \sqrt{d_{L-1}a_{L-2}},\dots, \sqrt{d_3 a_2} + \sqrt{d_{2}a_{1}}, \sqrt{d_2 a_1 } \}$. Therefore, for $\nu$ sufficiently large,
\[ 0 < \frac{ \theta_1(\nu)}{ \theta_i(\nu)} \leq \frac{A+B \sqrt{f(\nu)}}{C f(\nu) - D \sqrt{f(\nu)}},
\]
and so $ \limnu \theta_1(\nu) / \theta_i(\nu) =0 $ for $i=2,\dots,L$. Hence,
\[
\E T_{L,0}(\nu) \cdot \theta_1(\nu)=1+\sum_{i=2}^{L} \frac{\theta_1(\nu)}{ \theta_i(\nu)} \to 1, \quad \text{ as } \ninf,
\]
while for $ 2\leq  i \leq L$,
\[
\E T_{L,0}(\nu) \cdot \theta_i(\nu) > \frac{\theta_i(\nu)}{\theta_1(\nu)} \to \infty, \quad \text{ as } \ninf.
\]

\section{Proof of Lemma~\ref{lem:asymbounds}}
\label{ap2}
(a) The proof consists of a lower and an upper bound which asymptotically coincide. Indeed, using the bounds in Property (ii), one obtains that
\[
\liminfnu \frac{\E T(\nu)}{\E U(\nu)} \geq \liminfnu \frac{\E U(\nu) - \E V(\nu)}{\E U(\nu)} =1 \quad \text{ and } \quad \limsupnu \frac{\E T(\nu)}{\E U(\nu)} \leq \limsupnu \frac{\E U(\nu) + \E W(\nu)}{\E U(\nu)} =1. 
\]

(b) Once again, the proof consists of a lower and an upper bound which asymptotically coincide for all the continuity points of the tail distribution of $Z$, which will be denoted by $F(s)=\pr{Z > s}$.

For the lower bound, argue as follows. Property (i) implies that for any $\delta \in (0, 1)$, $\E W(\nu) \leq \delta \E U(\nu)$ for $\nu$ sufficiently large. Thus, using the lower bound in Property (iii), for $\nu$ sufficiently large,
\eqan{
\pr{\frac{T(\nu)}{\E T(\nu)} > t}
\geq & \, \pr{\frac{U(\nu) - V(\nu)}{\E U(\nu) + \E W(\nu)} > t} \nonumber \\
\geq & \, \pr{U(\nu) - V(\nu) > \E U(\nu) (1 + \delta) t} \nonumber \\
\geq & \, \pr{U(\nu) > \E U(\nu) (1+ 2\delta) t} - \pr{V(\nu) > \delta \E U(\nu) t}.\nonumber
}
Property (iii) implies that
\[
\limnu \pr{ \frac{U(\nu)}{\E U(\nu) } > (1+2\delta) t} = F \big((1+2\delta)\, t\big).
\]
Property (i) implies that for $\nu$ sufficiently large, $\E V(\nu) \leq \delta^2 \E U(\nu)$, so that
\[
\pr{V(\nu) > \delta \E U(\nu) t} \leq \frac{\E V(\nu)}{\delta \E U(\nu) t} \leq \frac{\delta}{t},
\]
by Markov's inequality. Taking liminf's, we obtain
\[
\liminfnu \pr{\frac{T(\nu)}{\E T(\nu)} > t} \geq F \big((1+2\delta) t\big) - \frac{\delta}{t}.
\]
Letting $\delta \downarrow 0$, we find
\eqn{ \label{eq:lowb}
\liminfnu \pr{\frac{T(\nu)}{\E T(\nu)} > t} \geq F\left(t \right).
}

For the upper bound, argue as follows. Property (i) implies that for any $\delta \in (0, 1)$, $\E V(\nu) \leq \delta \E U(\nu)$ for $\nu$ sufficiently large. Thus, using the upper bound in Property (iii), for $\nu$ sufficiently large,
\eqan{
\pr{\frac{T(\nu)}{\E T(\nu)} > t}
\leq &
\, \pr{\frac{U(\nu) + W(\nu)}{\E U(\nu) - \E V(\nu)} > t} \nonumber \\
\leq &
\, \pr{U(\nu) + W(\nu) > (1-\delta) \E U(\nu) t} \nonumber \\
\leq &
\, \pr{U(\nu) > (1-2\delta) \E U(\nu) t} + \pr{W(\nu) > \delta \E U(\nu)  t}.\nonumber
}
Property (iii) implies that
\[
\limnu \pr{ \frac{U(\nu)}{\E U(\nu) } > (1-2\delta) t} = F \big((1-2\delta) t\big).
\]
Property (i) implies that for $\nu$ sufficiently large, $\E W(\nu) \leq \delta^2 \E U(\nu)$, so that
\[
\pr{W(\nu) > \delta \E U(\nu) t} \leq \frac{\E W(\nu)}{\delta \E U(\nu) t} \leq \frac{\delta}{t},
\]
by Markov's inequality. Taking limsup's, we obtain
\[
\limsupnu \pr{\frac{T(\nu)}{\E T(\nu)} > t} \leq F\big((1-2\delta)\, t\big)+\frac{\delta}{t}.
\]
Letting $\delta \downarrow 0$, we find
\eqn{ \label{eq:uppb}
\limsupnu \pr{\frac{T(\nu)}{\E T(\nu)} > t} \leq F\left( t \right).
}
Combining~\eqref{eq:lowb} and~\eqref{eq:uppb} completes the proof.

\section{Proof of Lemma~\ref{lem:lt}}
\label{ap3}

Statement (a) is trivial, since $0\leq \LT_{T_k(\nu)/ \E T_k(\nu)}(s) \leq 1$ for all $s \in [0,\infty[$ and $\b_k = \limnu \E N_k(\nu) = 0$ for every $k \in \NA$. Statement (b) follows immediately after substituting $\b_k \in \R_+$ and $\g_k \in (0,1]$ in the limit, since $T_k (\nu) / \E T_k(\nu) \cd \rmexp(1)$. The proof of claim (c) is more involved and we need an auxiliary lemma.

Let $S_k(\nu):=\sum_{i=1}^{N_k(\nu)} T^{(i)}_k(\nu)$. From the integrability of $N_k(\nu)$ and $T_k(\nu)$ it follows that $\E S_k(\nu) = \E N_k(\nu) \cdot \E T_k(\nu)$. Consider the random variable 
\[
\tilde{S}(\nu)= g(\nu) \sum_{i=1}^{N_k(\nu)} \frac{T^{(i)}_k(\nu)}{\E T^{(i)}_k(\nu)} = g(\nu) \E N_k (\nu) \frac{S_k (\nu) }{\E S_k (\nu)}.
\] 
Since $N_k(\nu)$ has a geometric distribution, the Laplace transform of $S_k (\nu) /\E S_k (\nu)$ is given by
\[
\LT_{ S_k (\nu) /\E S_k (\nu)}(s)=G_{ N_k (\nu) } \left (\LT_{T_k(\nu)/ \E T_k(\nu)}(s / \E N_k (\nu)) \right )= \frac{1}{1+(1-\LT_{T_k(\nu)/ \E T_k(\nu)} (s / \E N_k (\nu)))\cdot \E N_k (\nu)},
\]
and hence
\eqn{\label{eq:stilde}
\LT_{\tilde{S}(\nu)}(s) = \LT_{S_k (\nu) /\E S_k (\nu)}(s g(\nu) \E N_k(\nu)) = \frac{1}{1+(1-\LT_{T_k(\nu)/ \E T_k(\nu)} (s g(\nu) ))\cdot \E N_k (\nu)}.
}
One can check that, if $k \in \SA$, then
\begin{itemize}
\item[(1)] $N_k(\nu) / \E N_k(\nu) \cd \rmexp(1)$, and 
\item[(2)] $\limnu \frac{\mathrm{Var} \left (T_k(\nu)/\E T_k(\nu)\right )}{ \E N_k(\nu)} = 0$.
\end{itemize}
Fact (1) is a standard result for geometric random variables which uses only the fact that $\limnu \E N_k(\nu) = \infty$ for $k \in \SA$. Moreover, since $T_k(\nu)$ is a first passage time of a birth-and-death process, using Corollary~4 in~\cite{JD08}, we can obtain explicitly the asymptotic order-of-magnitude of $\mathrm{Var} (T_k(\nu))$ and prove that $\limnu \mathrm{Var} (T_k(\nu)) / (\E T_k(\nu))^2 = 1$, from which claim (2) follows. 

These two facts and the technical lemma below imply that $S_k (\nu) /\E S_k (\nu) \cd \rmexp(1)$, and hence 
\[
\limnu \LT_{\tilde{S}(\nu)}(s) = \limnu \LT_{S_k (\nu) /\E S_k (\nu)}(s g(\nu) \E N_k(\nu)) = \frac{1}{1+s \g_k}.
\]
This fact, together with Equation~\eqref{eq:stilde}, implies that
\[
\limnu (1-\LT_{T_k(\nu)/ \E T_k(\nu)} (s g(\nu) ))\cdot \E N_k (\nu) = s \g_k,
\]
and the proof of statement (c) is concluded. We now state and prove the technical lemma mentioned above.
\begin{nnlem}
Assume that
\begin{itemize}
\item[{\rm (i)}] $\{X_i(\nu)\}_{i \geq 0}$ is a sequence of i.i.d.~integrable random variables, $X_i(\nu) \ed X(\nu)$, for every $\nu >0$,
\item[{\rm (ii)}] $N(\nu)$ is an integer-valued random variable, independent of all the $X_i(\nu)$'s and integrable for every $\nu >0$,
\item[{\rm (iii)}] $ N (\nu)/ \E N (\nu) \cd Z$ as $\ninf$, with $\pr{Z=0}=0$,
\item[{\rm (iv)}] $\limnu \frac{\mathrm{Var} \left (X(\nu)/\E X(\nu) \right )}{\E N(\nu)} =0$.
\end{itemize}
Define $S(\nu):=\sum_{i=1}^{N(\nu)} X_i(\nu)$. Then
\[
\frac{S(\nu)}{\E S(\nu)} \cd Z, \quad \text{ as } \ninf.
\]
\end{nnlem}
\begin{proof}
Firstly, Wald's identity guarantees that $S_{N}(\nu)$ is integrable and that $\E S (\nu) = \E N(\nu) \E X(\nu)$ for every $\nu >0$. Hence, without loss of generality, we can assume that $\E X (\nu) =1$ and study the asymptotic distribution of $S(\nu)/ \E N(\nu)$. Define $ S_n(\nu):=\sum_{i=1}^{n} X_i(\nu)$. Note that we can rewrite
\[
\frac{S(\nu)}{\E N(\nu)}=\frac{S(\nu)}{N(\nu)} \frac{N(\nu)}{\E N(\nu)}.
\]
Now we claim that
\eqn{\label{eq:cp}
\frac{S(\nu)}{N(\nu)} \stackrel{\mathbb P}{\to} 1, \quad \text{ as } \ninf.}
Indeed, for every $\epsilon >0$ we may write
\begin{align*}
\pr{ |S(\nu)/N(\nu) - \E X(\nu) |> \delta} &=\sum_{n=0}^{\infty} \pr{ |S(\nu)/N(\nu) - \E X (\nu) |> \delta, \, N(\nu)=n}\\
&=\sum_{n=0}^{\lfloor \epsilon \E N(\nu)\rfloor } \pr{\dots }+\sum_{n=\lfloor \epsilon \E N(\nu)\rfloor +1}^{\infty} \pr{\dots}\\
&\leq \sum_{n=0}^{\lfloor \epsilon \E N(\nu)\rfloor } \pr{N(\nu)=n}+\sum_{n=\lfloor \epsilon \E N(\nu)\rfloor +1}^{\infty}\pr{ |S(\nu)/N(\nu) - \E X(\nu) |> \delta, N(\nu)=n}\\
&\leq \pr{N(\nu)\leq \epsilon \E N(\nu)}+\sum_{n=\lfloor \epsilon \E N(\nu)\rfloor +1}^{\infty}\pr{|S_n(\nu)/n - \E X(\nu) |> \delta}\pr{ N(\nu)=n}
\end{align*}
Using Chebyshev's inequality, we find that the second term is bounded from above by
\[\sum_{n=\lfloor \epsilon \E N(\nu)\rfloor +1}^{\infty} \frac{\text{Var} (X(\nu))}{\delta^2 n} \pr{ N(\nu)=n} \leq \frac{\text{Var} (X(\nu))}{\epsilon \E N(\nu) \delta^2} \sum_{n=\lfloor \epsilon \E N(\nu)\rfloor +1}^{\infty}  \pr{ N(\nu)=n} \leq \frac{\text{Var} (X(\nu))}{\epsilon \E N(\nu) \delta^2}.
\]
For every $\xi >0$ and every $\epsilon >0$, there exists $\nu_{\epsilon,\xi} >0$ such that for $\nu > \nu_{\epsilon,\xi}$,
\[
\pr{N(\nu)\leq \epsilon \E N(\nu)} \leq \pr{ Z\leq \epsilon} + \xi \quad \text{ and } \quad \frac{\text{Var}X(\nu)}{\epsilon \E N(\nu)} \leq \xi.
\]
The first inequality follows from the fact that the c.d.f.~of $N(\nu) / \E N(\nu) $ converges pointwise to that of $Z$, by hypothesis~(i). On the other hand, the second inequality follows immediately from hypothesis~(ii). Therefore, for $\nu > \nu_{\epsilon,\xi}$,
\[
\pr{ |S(\nu)/N(\nu) - \E X (\nu) |> \delta} \leq  \pr{ Z\leq \epsilon} + \xi + \frac{\xi}{\delta^2}.
\]
Take e.g. $\epsilon = \delta$, $\xi=\delta^3$. Then for $\nu$ sufficiently large,
\[
\pr{ |S(\nu)/N(\nu) - \E X (\nu) |> \delta} \leq  \pr{ Z\leq \delta} + \delta^3 + \delta.
\]
Letting $\delta \downarrow 0 $ and using the fact that $\pr{Z=0}=0$, we obtain~\eqref{eq:cp}. Then hypothesis~(i) and Slutsky's theorem imply the conclusion. 
\end{proof}

\section{Proof of Lemma~\ref{lem:cond}}
\label{ap4}
In this subsection we use the same notation introduced in Section~\ref{sec8}, but without writing explicitly the dependence on $\nu$. Denote by $\pi_S$ the restriction of $\pi$ to $S$, i.e.~$\pi_S(A)=\pi(A \cap S)$, and define $\mu_S$ to be the equilibrium distribution conditioned on $S$, i.e.
\eqn{ \label{eq:mu}
\mu_S(A):=\frac{\pi_S(A)}{\pi(S)}.
}
Thanks to one of the characterizations of the total variation distance,
\[
\pi(S) \, || \mu_SP^t-\mu_S||_{\mathrm{TV}}=\pi(S) \sum_{\substack{y \in \OO :\\ \mu_S P^t(y) \geq \mu_S(y)}} \left[ \mu_SP^t(y)-\mu_S(y)\right].
\]
From~\eqref{eq:mu} it follows that $\pi_S P^t (y) = \pi(S) \mu_S P^t$ and that $\pi_S(y)= \pi(S) \mu_S(y)$. Moreover, $\mu_S P^t(y) \geq \mu_S(y) \Longleftrightarrow \pi_S P^t(y) \geq \pi_S(y)$. Therefore
\eqn{ \label{eq:pi}
\pi(S) \, || \mu_SP^t-\mu_S||_{\mathrm{TV}}= \sum_{\substack{y \in \OO :\\ \pi_S P^t(y) \geq \pi_S(y)}} \left[ \pi_SP^t(y)-\pi_S(y)\right].
}
Since $\pi_S(x) >0$ iff $x\in S$ and $\pi_S(x)=\pi(x)$ if $x\in S$,
\[
\pi_S P^t(y)=\sum_{x \in \OO} \pi_S(x) P^t(x,y)=\sum_{x \in S} \pi(x) P^t(x,y) \leq \sum_{x \in \OO} \pi(x) P^t(x,y)=\pi(y).
\]
From this inequality it follows that $\pi_S P^t(y) \leq \pi_S(y)$ for $y \in S$ and trivially $\pi_S P^t(y) \geq \pi_S(y)$ for $y \in S^c$, since the right-hand side is zero. Using these last two inequalities, identity~\eqref{eq:pi} can be rewritten as
\[
\pi(S) \, || \mu_SP^t-\mu_S||_{\mathrm{TV}}= \sum_{y \in S^c} \left[ \pi_SP^t(y)-\pi_S(y)\right]= \sum_{y \in S^c} \pi_S P^t(y),
\]
where the last equation follows from the observation that $\pi_S(y)=0$ when $y \in S^c$. Furthermore
\eqn{ \label{eq:star}
\sum_{y \in S^c} \pi_S P^t(y)= \sum_{y \in S^c} \sum_{x \in S} \pi(x) \pr{\xt= y ~|~ X(0)=x }=\sum_{x \in S}  \pr{ \xt \in S^c, \, X(0)=x }.
}
Define $N_{x \to y}(t)$ as the number of transitions from state $x$ to state $y$ during the time interval $[0,t]$, so that
\eqn{\label{eq:ineq}
\pr{ \xt \in S^c, \, X(0)=x } \leq \sum_{x' \in S, y' \in S^c} \pr{ N_{x' \to y'}(t) \geq 1, \, X(0)=x } \leq \sum_{x' \in S, y' \in S^c} \EE{ N_{x' \to y'}(t), \, X(0)=x },
}
and hence, substituting~\eqref{eq:ineq} in~\eqref{eq:star}
\eqan{
\sum_{y \in S^c} \pi_S P^t(y) &\leq \sum_{x \in S } \sum_{x' \in S, y' \in S^c} \EE{ N_{x' \to y'}(t) , \, X(0)=x } \nonumber \\
&=\sum_{x' \in S, y' \in S^c} \EE{ N_{x' \to y'}(t) } \nonumber \\
&=\sum_{x' \in S, y' \in S^c} \int_{u=0}^t \pr{X(u)=x'} \, q(x',y') du \nonumber \\
&\leq \sum_{x' \in S, y' \in S^c} \pi (x') \, q(x',y') t \nonumber \\
&=t \, Q(S,S^c).\nonumber
}
The second last passage follows from the fact that $\pr{X(u)=x'}=\pi_S P^u(x') \leq \pi_S(x')=\pi(x')$. Then
\[
||\mu_S P^t -\mu_S||_{\mathrm{TV}} \leq t \, \Phi(S).
\]
Assuming that $\pi(S) \leq r$, $r \in [0,1]$, one has that
\[
||\mu_S - \pi||_{\mathrm{TV}} = \max_{A \subset \OO} |\mu_s(A) -\pi(A) | \geq \pi(S^c)-\mu(S^c)=\pi(S^c)=1-\pi(S) \geq 1-r.
\]
The triangular inequality implies that
\[
1-r \leq ||\mu_S - \pi||_{\mathrm{TV}} \leq ||\mu_S  -\mu_S P^t ||_{\mathrm{TV}} +||\mu_S P^t - \pi||_{\mathrm{TV}}.
\]
We know that $\bar{d}(t) \leq 2 d(t)$ and hence for $t=\tm(\e)$, $||\mu_S P^t - \pi||_{\mathrm{TV}} \leq \bar{d}(\tm(\e)) \leq 2 d(\tm(\e)) = 2 \e$. Taking $t=\tm(\e)$, it follows that
\[
1-r \leq \tm(\e) \Phi(S) + 2 \e.
\]
Rearranging and minimizing over $S$ concludes the proof.
\end{document}